\documentclass[12pt]{amsart}
\usepackage{amssymb,amsmath,amsthm,amssymb}
\usepackage{mathtools}
\usepackage{graphicx}
\usepackage{comment}
\usepackage{hyperref}
\usepackage{enumerate}
\usepackage{color}
\usepackage{epstopdf}
\renewcommand{\epsilon}{\varepsilon}

\DeclareMathOperator{\spa}{span}

\DeclareMathOperator{\arc}{arc}
\def\R{\mathbb{R}}
\def\E{\mathbb{E}}

\def\N{\mathbb{N}}
\def\Z{\mathbb{Z}}

\def\a{\alpha}
\def\e{\epsilon}

\def\r{\rho}

\def\o{\omega}
\def\l{\lambda}
\def\k{\kappa}

\def\wt{\widetilde}

\def\la{\langle}
\def\ra{\rangle}

\newtheorem{theorem}{Theorem}[section]
\newtheorem{lemma}[theorem]{Lemma}

\newtheorem{thm}[theorem]{Theorem}
\newtheorem{remark}[theorem]{Remark}
\newtheorem{corollary}[theorem]{Corollary}

\newtheorem{claim}{Claim}[theorem]

\newtheoremstyle{TheoremNum}
        {\topsep}{\topsep}              %%% space between body and thm
        {\itshape}                      %%% Thm body font
        {}                              %%% Indent amount (empty = no indent)
        {\bfseries}                     %%% Thm head font
        {.}                             %%% Punctuation after thm head
        { }                             %%% Space after thm head
        {\thmname{#1}\thmnote{ \bfseries #3}}%%% Thm head spec
    \theoremstyle{TheoremNum}

\title[Collapsing ancient solutions of MCF]{Collapsing ancient solutions of mean curvature flow}
\author{Theodora Bourni}
\author{Mat Langford}
\thanks{The second author was partially supported by an Alexander von Humboldt fellowship.}
\author{Giuseppe Tinaglia}
\thanks{The third author was partially supported by EPSRC grant no. EP/M024512/1}

\address{Department of Mathematics, University of Tennessee Knoxville, Knoxville TN, 37996-1320}
\email{tbourni@utk.edu}
\email{mlangford@utk.edu}
\address{Department of Mathematics, King's College London, London, WC2R 2LS, U.K.}
\email{giuseppe.tinaglia@kcl.ac.uk}

\begin{document}

\begin{abstract}
We construct a compact, convex ancient solution of mean curvature flow in $\R^{n+1}$ with $O(1)\times O(n)$ symmetry that lies in a slab of width $\pi$. We provide detailed asymptotics for this solution and show that, up to rigid motions, it is the only compact, convex, $O(n)$-invariant ancient solution that lies in a slab of width $\pi$ and in no smaller slab. %The existence of such a solution was suggested, based on numerical evidence, by Angenent, Daskalopoulos and \u Se\u sum. 

%The solution we construct is rotationally symmetric. As time approaches minus infinity, the maximum of its mean curvature converges to one and, after translating, its `edge' regions converge to grim planes of width $\pi$. We show that the latter holds for any compact convex rotationally symmetric ancient solution which lies in a slab of width $\pi$ (and in no smaller slab). %Using well-known estimates of Andrews, we also construct an ancient solution of the Gauss curvature flow in $\R^3$ with the same geometric properties.
\end{abstract}

\maketitle

\tableofcontents

\section{Introduction}

A smooth one parameter family $\{\Sigma^n_t\}_{t\in I}$ of smooth hypersurfaces $\Sigma_t^n$ of $\R^{n+1}$ is a solution of mean curvature flow if
\[
\frac{\partial F}{\partial t}(x,t)=\vec H(x,t)\quad\text{for all}\quad (x,t)\in \Sigma^n\times I
\]
for some smooth one parameter family $F:\Sigma^n\times I\to\R^{n+1}$ of smooth immersions $F(\cdot,t):\Sigma^n\to \R^{n+1}$ with $\Sigma^n_t=F(\Sigma^n,t)$, where $\vec H(\cdot,t)$ is the mean curvature vector of $F(\cdot,t)$. We say that a solution $\{\Sigma^n_t\}_{t\in I}$ of mean curvature flow is compact/convex/etc if this is true for each time slice $\Sigma^n_t$. %Here $\{F(\cdot,t)\}_{t<0}$ is a smooth one-parameter family of smooth, convex embeddings of $S^n$ into $\R^{n+1}$ with mean curvature vector $\vec H(\cdot,t)$. %In Section \ref{sec:Gauss}, we will also consider the Gauss curvature flow, in which case the mean curvature vector is replaced by the inward normal vector times the Gauss curvature. 
We shall be interested in solutions which are defined on time intervals of the form $I=(-\infty,T)$, where $T\leq \infty$. Such solutions are referred to as \emph{ancient} because they have existed for an infinite time in the past. Furthermore, we will mostly be interested in compact solutions so that, without loss of generality, $T=0$ is the maximal time of existence. The most prominent example of a compact ancient solution of mean curvature flow is the shrinking sphere $S^n_{\sqrt{-2nt}}$. 

A great deal of interest in ancient solutions has arisen through their natural role in the study of singularities of mean curvature flow.
%
%In this paper, we construct a compact, convex, rotationally symmetric ancient solution of mean curvature flow in $\R^{n+1}$ that lies in a slab $\Omega:=\{(x,\hat x)\in \R\times\R^n\cong \R^{n+1}:|x|<\frac{\pi}{2}\}$. The solution is also reflection symmetric with respect to the mid-plane of the slab. 
%
%In contrast to the methods used in~\cite{Wa11} by Wang, our construction is explicit and the geometric properties of the solutions overtly present themselves. Indeed, we then prove this is the unique, compact, rotationally symmetric, convex ancient solution of mean curvature flow in $\R^3$ that lies in a slab of width $\pi$. This is an example of a collapsing ancient solution of mean curvature flow in $\R^3$. See Section \ref{Angenent ovals} for a definition of collapsing. 
%
%The solution we construct is the limit of a sequence of mean curvature flows obtained by evolving rotated time slices of the Angenent oval. As time goes to minus infinity, its maximum mean curvature converges to one, its `middle' region converges to the boundary of the slab and, after translating, its `edge' converges to a grim plane asymptotic to the boundary of the slab (see \S \ref{sec:uniqueasymptotics}).
Moreover, they tend to exhibit rigidity phenomena analogous to those of complete minimal surfaces; for example, when $n\geq 2$, under certain geometric conditions ---  uniform convexity, bounded eccentricity, type-I curvature decay or bounded isoperimetric ratio, for instance --- the only compact, convex ancient solutions are shrinking spheres \cite{HuSi15} (see also \cite{DHS,HaHe,L}). When the ambient space is the sphere, the result is even nicer: In that case, the only geodesically convex ancient solutions are shrinking hemispheres \cite{BrLo,HuSi15}. The shrinking spheres are not the only such solutions, however --- there exists a family of compact, convex ancient solutions which contract to round points as $t$ goes to zero but become more eccentric as $t$ goes to minus infinity, resembling a shrinking cylinder $\R^k\times S^{n-k}_{\sqrt{-2(n-k)t}}$ in the `parabolic' region and a convex translating solution in the `tip' region \cite{Ang12,HaHe,Wh03}. We note that these examples are non-collapsing (in the sense of \cite[Section 3]{ShWa09} and \cite{An12}); that is, each point is tangent to an enclosed ball of radius comparable (uniformly in time) to one over the mean curvature at that point.% and their `blow-downs' are the shrinking cylinders $\R^k\times S^{n-k}_{\sqrt{-2(n-k)t}}$. This agrees with a result of Wang \cite[Corollary 6.3]{Wa11}, which shows that shrinking cylinders (including the shrinking sphere) and planes of multiplicity two are the only possible blow-downs. 

For curves evolving in the plane, there is a compact, convex ancient solution of curve shortening flow, known as the \emph{Angenent oval} or the \emph{paperclip}, which lies in a strip region of width $\pi$ \cite{Ang92}. In particular, this solution is collapsing. As $t$ goes to minus infinity, the Angenent oval tends to the boundary of the strip, whereas, after translating one of its two points of maximal displacement to the origin, it resembles the translating Grim Reaper solution. Modulo rigid motions, time translations and parabolic dilations, the shrinking circle and the Angenent oval are the only compact, embedded, convex ancient solutions of curve shortening flow \cite{DHS}. In particular, the shrinking circles are the only convex ancient solutions which are non-collapsing. In higher dimensions, the classification of convex, compact ancient solutions remains a open problem, even for non-collapsing solutions. We refer the reader to \cite{ADS} for some recent progress in this direction. %The existence of collapsing solutions in higher dimensions has been suggested by Angenent, Daskalopoulos and \u Se\u sum based on numerical evidence \cite{ADS} but until now no such solution had been rigorously constructed.

In higher dimensions, Xu-Jia~Wang has constructed compact, convex ancient solutions in $\R^{n+1}$ which lie in slab regions by taking a limit of solutions of the Dirichlet problem for the level set flow \cite{Wa11}. %; however, he did not study their asymptotics in detail. ******************************

In this paper, we will provide a detailed construction of an $O(1)\times O(n)$-invariant solution of mean curvature flow, including a precise description of its asymptotics. Our methods are rather different from Wang's, however; indeed, we emphasize elementary geometric techniques throughout and make no use of the level set flow (our solution is instead the limit of a sequence of mean curvature flows obtained by evolving rotated time slices of the Angenent oval). 

\begin{thm}\label{thm:existence}
There exists a compact, convex, $O(1)\times O(n)$-invariant ancient solution $\{\Sigma^n_t\}_{t\in(-\infty,0)}$ of mean curvature flow in $\R^{n+1}$ which lies in the slab $\Omega:=\{x\in \R^{n+1}:\vert x_1\vert<\frac{\pi}{2}\}$ and has the following properties.
\begin{enumerate}
\item[(1)] $\{\lambda \Sigma_{\lambda^{-2}t}\}_{t\in(-\infty,0)}$ converges uniformly in the smooth topology to the shrinking sphere $S^n_{\sqrt{-2nt}}$ as $\lambda\to 0$,
\item[(2)] $\{\Sigma_{t+s}\}_{t\in(-\infty,-s)}$ converges locally uniformly in the smooth topology to the stationary solution $\partial\Omega$ as $s\to-\infty$ and
\item[(3)] for any unit vector $\varphi\in \{e_1\}^\perp$, $\{\Sigma_{t+s}-P(\varphi,s)\}_{t\in (-\infty,-s)}$ converges locally uniformly in the smooth topology as $s\to-\infty$ to the Grim hyperplane which translates with unit speed in the direction $\varphi$, where, given any $v\in S^n$, $P(v,t)$ denotes the unique point of $\Sigma^n_t$ with outward pointing unit normal $v$.
\end{enumerate}
Moreover, as $t\to -\infty$,
\begin{enumerate}
%\item[2\,a)] $\displaystyle\min_{\Sigma_t}H=H(P(e_1,t))\leq o\left(\tfrac{1}{(-t)^{k}}\right)$ for any $k>0$,
\item[(4)] $\displaystyle\min_{p\in \Sigma_t}\vert p\vert=\vert{P(e_1,t)}\vert \geq \tfrac{\pi}{2}-o\left(\tfrac{1}{(-t)^k}\right)$ for any $k>0$ and
%\item[2\,c)] $\displaystyle\max_{\Sigma_t} H=H(P(\varphi,t))\geq \left(1+\tfrac{n-1}{-t}+o\left(\tfrac{1}{(-t)^{2-\varepsilon}}\right)\right)$ for any unit vector $\varphi\in \{e_1\}^\perp$ and any $\varepsilon>0$, and 
\item[(5)] $\displaystyle\max_{p\in \Sigma_t}\vert p\vert= \vert{P(\varphi,t)}\vert=-t+(n-1)\log(-t)+C+o(1)$ for any unit vector $\varphi\in \{e_1\}^\perp$, where $C\in\R$ is some constant.
\end{enumerate}
\end{thm}

In fact, we are able to say even more about the asymptotics of this solution, including asymptotics for the speed; see Lemma \ref{lem:uHhk}, Corollary \ref{cor:improvedHestimate} and Remark \ref{rem:betterasymptotics}.

We note that the convergence to a `round point' in item (1) is a consequence of Huisken's theorem \cite{Hu84} and well-known arguments show that the `parabolic' region converges to the boundary of the slab, as in item (2); see Lemma \ref{lem:asymptoticsmiddle}. With regards to item (3), well-known arguments also show that the `edge' region converges to a Grim hyperplane, at least along a subsequence of times (see Lemma \ref{lem:grimconvergence}); however, it is non-trivial to rule out limit Grim hyperplanes which are smaller than the one asymptotic to the boundary of the slab. This is the content of Corollary \ref{cor:asymptoticsedge}. The remaining asymptotics are derived in Sections \ref{sec:area estimates} and \ref{sec:uniqueness}. In fact, we actually show that \emph{any} compact, convex, $O(n)$-invariant ancient solution contained in the slab $\Omega$ (and no smaller slab) satisfies the asymptotics (1)-(5)\footnote{Although it must be noted that the approximating solutions used to construct the particular solution described in Theorem \ref{thm:existence} are used in a crucial way to obtain finiteness of the constant $C$ in statement (5). See Section \ref{sec:uniqueness}.}. By applying an Alexandrov reflection argument, we are then able to obtain the following uniqueness result.

%Combined with a novel application of the Alexandrov reflection principle, this allows us to obtain the following uniqueness result.

\begin{thm}\label{thm:uniqueness}
Let $\{\Sigma_t\}_{t\in(-\infty,0)}$ be a compact, convex, $O(n)$-invariant ancient solution of mean curvature flow in $\R^{n+1}$ which lies in a slab $\Omega_{e,\alpha}:=\{x\in \R^{n+1}: \vert{x\cdot e}\vert<\alpha\}$ for some $e\in S^n$ and $\alpha>0$ and in no smaller slab. Then, after a rigid motion and a parabolic rescaling, $\{\Sigma_t\}_{t\in(-\infty,0)}$ coincides with the solution constructed in Theorem \ref{thm:existence}.
\end{thm}

It is worth noting that reflection symmetry is not assumed in Theorem \ref{thm:uniqueness}. Moreover, by a result of Wang, it even suffices to assume that the solution only lies in a halfspace rather than a slab \cite[Corollary 2.1]{Wa11}.

Finally, we remark that these arguments apply (and, indeed, are significantly simplified) in case $n=1$. So our methods also suggest a new approach to the classification of compact, convex ancient solutions of the curve shortening flow \cite{DHS}.

%Finally, we invite the reader to check that, in case $n=1$, the $O(n)$-symmetry hypothesis is not required in the proof of Theorem \ref{thm:uniqueness}. (Moreover, the arguments are greatly simplified in this case.) Thus, Theorem \ref{thm:uniqueness} provides a novel approach to the classification of compact, convex ancient solutions of the curve shortening flow \cite{DHS}. \texttt{This would appear to require simply an additional application of the Alexandrov reflection principle}.

%It is worth noting that reflection symmetry of the solution is not assumed in Theorem \ref{thm:uniqueness}. In fact, it even suffices to assume that the solution only lies in a halfspace rather than a slab --- in the final section of the paper, we show that any compact, convex ancient solution that lies in a halfspace actually lies in a slab.

%\begin{thm}\label{thm:halfspace}
%Let $\{\Sigma_t\}_{t\in(-\infty,0)}$ be a compact, convex ancient solution of the mean curvature flow in $\R^{n+1}$ which, for some $e\in S^n$ and $\alpha\in \R$, lies in the halfspace $\Pi_{e,\alpha}:=\{x\in \R^{n+1}: x\cdot e<\alpha\}$. Then $\{\Sigma_t\}_{t\in(-\infty,0)}$ lies in the slab $\Omega_{e,\alpha,\beta}:=\{x\in \R^{n+1}: \beta< x\cdot e<\alpha\}$ for some $\beta\in \R$.
%\end{thm}

\section*{Acknowledgements}

We are indebted to Ben Andrews, Sigurd Angenent and Panagiota Daskalopoulos for many helpful suggestions on the problem and to Paul Bryan and Mohammad Ivaki for providing valuable feedback on an early draft of the paper.

\section{The Angenent oval}\label{Angenent ovals}

We begin by reviewing some geometric properties of convex, closed curves and the Angenent oval that will be important in our construction. In this section, and throughout the paper, we will regularly identify $S^1\cong \R/2\pi\Z$.

First recall that any positive $2\pi$-periodic function $\kappa\in C^0(\R)$ satisfying
\[
\int_0^{2\pi} \frac{\cos\theta}{\kappa(\theta)}d\theta=0\quad\text{and}\quad \int_0^{2\pi} \frac{\sin\theta}{\kappa(\theta)}d\theta=0
\] 
defines a simple, closed, convex $C^2$ planar curve $\gamma:S^1\to\R^2$ via the formula %\cite{GaHa86}
\begin{equation}\label{eq:curvedef}
\gamma(\theta):=\left(\int_0^{\theta} \frac{\cos u}{\kappa(u)}du\,,\int_{\frac{\pi}{2}}^{\theta} \frac{\sin u}{\kappa(u)}du\right)\,.
\end{equation}
Note that
\[
\gamma'(\theta)=\left(\frac{\cos\theta}{\kappa(\theta)},\frac{\sin\theta}{\kappa(\theta)}\right)\implies \tau(\theta):=\frac{\gamma'(\theta)}{|\gamma'(\theta)|}=(\cos\theta,\sin\theta)\,,
\]
so $\theta$ has the geometric interpretation as the \emph{turning angle} of $\gamma$; that is, the counter-clockwise angle between the $x$-axis and the tangent vector $\tau$. Moreover,
\[
s(\theta):=\int_0^\theta\frac{du}{\kappa(u)}
\]
defines an arc-length parameter so that
\[
\partial_s\tau=-\kappa\,\partial_\theta \tau=-\kappa\nu\,,
\]
where
\[
\nu:=(\sin\theta,-\cos\theta)
\]
is the outward-pointing unit normal. So $\kappa$ corresponds to the curvature of the curve. Conversely, up to a translation in $\R^2$, any simple, closed, convex, $C^2$ planar curve can be written in the form \eqref{eq:curvedef} by parametrizing with respect to turning angle.

The \emph{Angenent oval} is the smooth one-parameter family of curves\footnote{We will sometimes refer to a particular time slice $\gamma(\cdot,t)$ as an Angenent oval.} $\gamma(\cdot,t):[0,2\pi)\to\R^2$, $t\in(-\infty,0)$, defined by \eqref{eq:curvedef} with curvature given by the formula
\begin{equation}\label{k}
\k^2(\theta,t)=\frac{1}{\mathrm{e}^{-2t}-1}+\cos^2\theta\,.
\end{equation}
It is readily verified that \cite{DHS}
\[
\kappa_t=\kappa^2(\kappa_{\theta\theta}+\kappa)\,,
\]
where the subscripts denote corresponding partial derivatives. Using also the fact that $\kappa_\theta(0,t)\equiv 0$, we find
\[
\partial_t\gamma=-\kappa\nu-\kappa_\theta\tau\,.
\]
Thus, up to a tangential reparametrization, the Angenent oval is an ancient solution of the curve shortening flow.

Using \eqref{k}, we can compute the $x$ and $y$ coordinates of the Angenent oval explicitly: Setting $a^2(t):=\frac{1}{\mathrm{e}^{-2t}-1}$, we find
\begin{align*}%\label{eq:Angenentx}
x(\theta,t)=\int_{0}^\theta\frac{\cos u}{\sqrt{\cos^2u +a^2(t)}}du
%={}&\left.\arctan\left(\frac{\sin u}{\sqrt{\cos^2u+a^2(t)}}\right)\right|_0^{\theta}\\
={}&\arctan\left(\frac{\sin\theta}{\sqrt{\cos^2\theta+a^2(t)}}\right)
\end{align*}
and
\begin{align*}%\label{eq:Angenenty}
y(\theta, t)={}&\int_{\tfrac{\pi}{2}}^\theta\frac{\sin u}{\sqrt{\cos^2u+a^2(t)}}du\nonumber\\
%={}&-\left.\log\left(\cos u + \sqrt{\cos^2u+a^2(t)}\right)\right|_{\frac\pi2}^\theta\\
={}&\log\left(\frac{a(t)}{\sqrt{\cos^2\theta+a^2(t)}+\cos\theta}\right)\nonumber\\
={}&-t+\log\left(\frac{\sqrt{\cos^2\theta+a^2(t)}-\cos\theta}{\sqrt{a^2(t)+1}}\right)\,.
\end{align*}
In particular, %it is straightforward to check that
\begin{align}\label{eq:Angenentimplicit}
\cos x=\mathrm{e}^t\cosh y\,.
\end{align}
From \eqref{eq:Angenentimplicit}, it is easily seen that $\gamma(\cdot,t)$ is reflection symmetric with respect to both the $x$-axis and the $y$-axis. %Moreover, one can readily verify that the Angenent oval satisfies curve shortening flow (up to a time-dependent reparametrization) by checking that the function
%\[
%u(x,y):=\log\left(\frac{\cos x}{\cosh y}\right)
%\]
%satisfies the level set flow
%\[
%\dvg\left(\frac{Du}{\vert Du\vert}\right)=-\frac{1}{\vert Du\vert}\,.
%\]

Note now that $\k^2(\cdot,t)$ attains its minimum value, $\frac{1}{1-\mathrm{e}^{2t}}-1$, at $\theta=\frac{\pi}{2}$ and its maximum value, $\frac{1}{1-\mathrm{e}^{2t}}$, at $\theta=\pi$ and is strictly increasing in the interval $(\tfrac{\pi}{2},\pi)$. In particular, the vertex set of $\gamma(\cdot,t)$ is $\{0,\tfrac{\pi}{2},\pi,\tfrac{3\pi}{2}\}$ for all $t<0$. Define the horizontal and vertical displacements
\[
h(t):=\max_{\theta \in S^1}x(\theta,t)= x\left(\tfrac\pi 2,t\right)\quad \text{and} \quad \ell(t):=\max_{\theta \in S^1}y(\theta,t)=y(\pi,t)\,.
\]

\begin{figure}
\begin{center}
\includegraphics[width=0.7\textwidth]{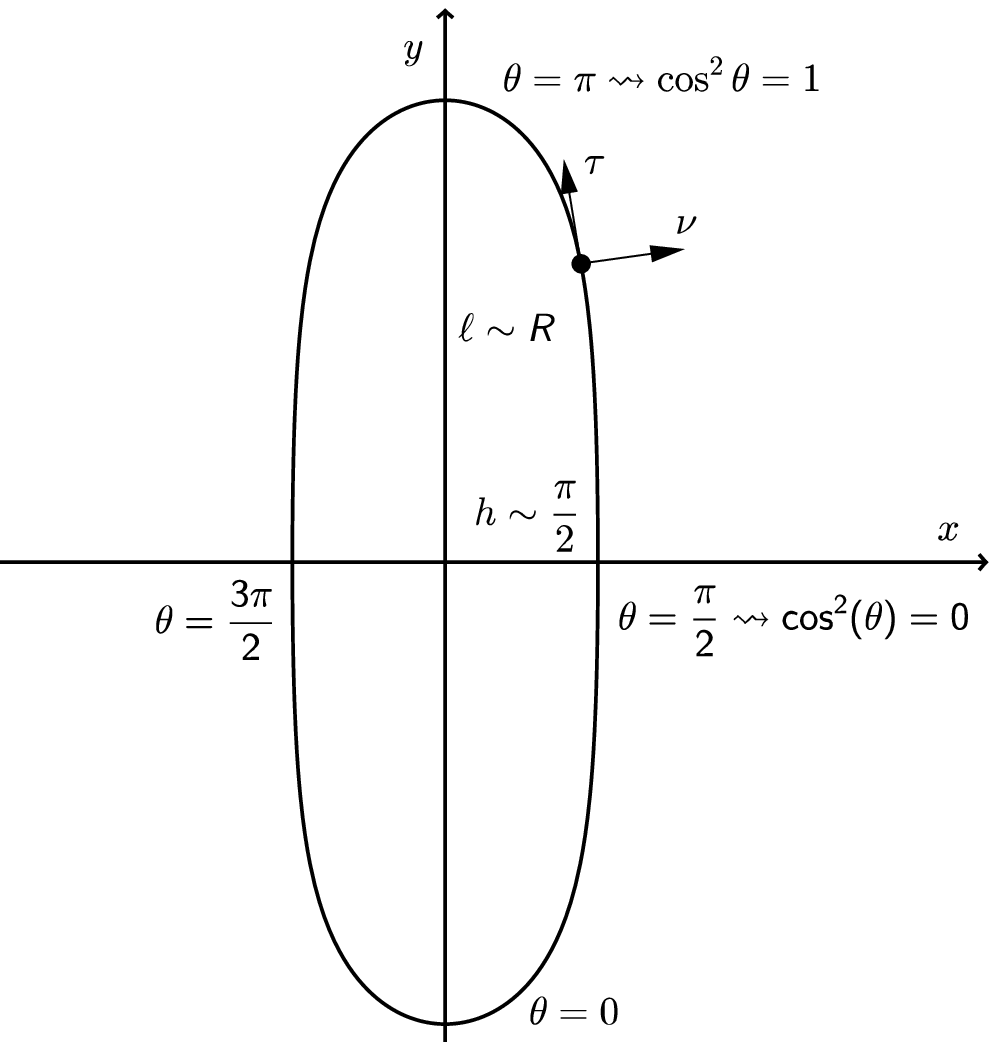}% picture filename
\end{center}
\caption{The Angenent oval at time $t=-R$.}\label{fig1}
\end{figure}

\begin{lemma}\label{CR-lem}
For every $t<0$,
\[
\frac{\pi}{2}\left(1-\mathrm{e}^t\right)\le h(t)\le \frac{\pi}{2}\quad \text{and}\quad -t\le \ell(t)\le -t+\log 2.
\]
\end{lemma}
\begin{proof}
We use \eqref{eq:Angenentimplicit} to compute
\begin{align*}
\cos h(t)=\mathrm{e}^t\quad\text{and}\quad\cosh \ell(t)=\mathrm{e}^{-t}\,.
\end{align*}
The claimed estimates follow: Clearly $h(t)<\frac{\pi}{2}$. Estimating $x\leq \sin\left( \tfrac{\pi}{2}x\right)$ for $x\in [0,1]$ yields the crude estimate
\begin{align*}
\cos h(t)=\mathrm{e}^t\leq \sin \left(\tfrac{\pi}{2}\mathrm{e}^t\right)=\cos\left(\tfrac{\pi}{2}-\tfrac{\pi}{2}\mathrm{e}^t\right)\,,
\end{align*}
which yields the lower bound for $h$. To estimate $\ell$ from above, we simply observe that
\begin{align*}
\mathrm{e}^{-t+\log 2}=2\mathrm{e}^{-t}=\mathrm{e}^{\ell(t)}+\mathrm{e}^{-\ell(t)}\geq \mathrm{e}^{\ell(t)}\,.
\end{align*}
To obtain the lower bound for $\ell$, we crudely estimate
\begin{align*}
\cosh\ell(t)=\mathrm{e}^{-t}\geq \cosh(-t)
\end{align*}
and recall that $\cosh x$ is increasing for $x\geq 0$.
\end{proof}

In fact, the Angenent oval lies inside the two vertically translating Grim Reapers which reach the origin at time $t=\log 2$

\begin{lemma}\label{lem:grimavoidance}
For every $t<0$, $\Gamma_t\cap G^{\pm}_t=\emptyset$, where
\[
G^\pm_t:=\{(x,\pm(-t+\log 2+\log\cos x)):x\in(-\tfrac{\pi}{2},\tfrac{\pi}{2})\}\,.
\]
\end{lemma}
\begin{proof}
If the claim did not hold, then, by \eqref{eq:Angenentimplicit}, there would be a point $y\in \R$ satisfying
\begin{align*}
y=-t+\log 2+\log(\mathrm{e}^t\cosh y)\quad\implies\quad \mathrm{e}^y=2\cosh y>\mathrm{e}^y\,,
\end{align*}
an impossibility.
\end{proof}

%Finally, we recall that an embedded, mean convex solution $F:S^n\times I\to \R^{n+1}$ of mean curvature flow is \emph{collapsing} (in the sense of \cite[Section 3]{ShWa09}) if $\inf_{(x,t)\in S^n\times I}\overline r(x,t)H(x,t)=0$, where $\overline r(x,t)$ is the \emph{inscribed radius} at $(x,t)$ (the radius of the largest ball enclosed by $F(S^n,t)$ which is tangent at $F(x,t)$). Then the Angenent oval is collapsing, since $\overline r(x,t)<\frac{\pi}{2}$ for all $(x,t)\in S^1\times(-\infty,0)$ but $\kappa(\tfrac{\pi}{2},t)$ tends to $0$ as $t$ tends to $-\infty$.

\section{\texorpdfstring{$O(n)$}{\textit{O(n)}}-invariance}\label{sec:rotatingtheovals}

%We will need some properties of $O(n)$-invariant hypersurfaces in $\R^{n+1}$.
A convex hypersurface $\Sigma^n\hookrightarrow \R^{n+1}$ is $O(n)$-\emph{invariant} with respect to some unit vector $e_1\in \R^{n+1}$ if $\Sigma^n$ is invariant under the action of $O(n)$ by rotation in the $\{e_1\}^\perp$ hyperplane; more explicitly, an element $h\in O(n)$ acts on a point $p\in \R^{n+1}$ by
\[
p\mapsto \left\langle p,e_1\right\rangle e_1+h\cdot\left(p-\left\langle p,e_1\right\rangle e_1\right)\,,
\]
where $\cdot$ denotes the standard action of $O(n)$ on the $n$-dimensional subspace $\{e_{1}\}^\perp$. 

Given such a hypersurface $\Sigma^n\hookrightarrow \R^{n+1}$, fix any orthogonal unit vector $e_2\in \{e_{1}\}^\perp$ and set $\E^2:=\spa\{e_1,e_2\}$. For convenience, we will use coordinates $(x,y,z):\R^{n+1}\to \R\times \R\times \R^{n-1}$ for  $\R^{n+1}$ defined by
\[
p=x(p)e_1+y(p)e_2+\sum_{i=3}^{n+1}z^i(p)e_i\,,
\]
where $\{e_i\}_{i=3}^{n+1}$ is a basis for $\{e_1,e_2\}^\perp$. Observe that the profile curve $\Gamma:= \Sigma^n\cap \E^2$ is smooth since it is geodesic in $\Sigma^n$.
%Let $\Gamma\subset \R^2$ be a simple, smooth closed curve in the plane which is symmetric under reflections in the line $\R e$ defined by some unit vector $e\in S^1$. Identifying $\R^2$ with a subspace of $\R^{n+1}$ in such a way that $e$ corresponds to $e_{n+1}$, consider the hypersurface of revolution $\Sigma^n\subset \R^{n+1}$ defined by
%\[
%\Sigma^n=\{(o(\hat x),x_{n+1}):(\hat x,x_{n+1})\in \Gamma\},
%\]
%where $o(p)$ denotes the orbit of $p\in \R^{n}\times \{0\}\cong \R^n$ by the standard action of the special orthogonal group $SO(n)$. 
If we parametrize $\Gamma$ with respect to turning angle by a curve $\gamma:S^1\to \E^2$ then we may parametrize $\Sigma^n$ in polar coordinates $(\theta,\varphi)\in S^{1}\times S^{n-1}$ by
\begin{align*}
F(\theta,\varphi):=x(\theta)e_1+ y(\theta)\varphi\,,
\end{align*}
where we define $(x(\theta),y(\theta))$ by $\gamma(\theta)=x(\theta)e_1+y(\theta)e_2$ and identify $S^{n-1}$ with the unit sphere in $\{e_1\}^\perp$. We claim that the mean curvature of $\Sigma^n$ can be expressed, purely in terms of $\theta$, by
\begin{equation}\label{eq:rotationalH}
H(\theta)=\kappa(\theta)+(n-1)\lambda(\theta)\,,
\end{equation}
where $\kappa$ is the curvature of $\gamma$ and $\lambda$ is its `rotational curvature', defined by
\[
\lambda(\theta)=-\frac{\cos\theta}{y(\theta)}\,.
\]
Indeed, with respect to this parametrization, the Weingarten map takes the form
\[
W_{(\theta,\varphi)}=\left[\begin{array}{cc} \kappa & 0\\ 0& \lambda \mathrm{I}_{n-1\times n-1}\end{array}\right]\,.
\]

We note that $\lambda$ is well-defined and smooth.

\begin{lemma}\label{lem:lambdalemma}
Let $\gamma:S^1\to \R^2$ be a smooth, closed,  convex curve which is reflection symmetric in the $x$-axis and parametrized by turning angle. Set
\begin{equation*}
\lambda(\theta):=\begin{dcases}-\frac{\cos\theta}{y(\theta)}&\text{ when}\quad \theta\in S^1\setminus\{\tfrac{\pi}{2},\tfrac{3\pi}{2}\}\\
\kappa(\theta)& \text{ when} \quad \theta\in\{\tfrac{\pi}{2},\tfrac{3\pi}{2}\}\,,
\end{dcases}
\end{equation*}
where 
\[
y(\theta)=\langle\gamma(\theta,t), e_2\rangle=\int^\theta_{\frac{\pi}{2}}\frac{\sin u\,du}{\kappa(u)}\,.
\]
Then $\lambda\in C^\infty(S^1)$. Moreover, at the poles $\theta=\pm\tfrac{\pi}{2}$,
\[
\lambda=\kappa\quad\text{and}\quad\lambda_\theta=\kappa_\theta=0\,,%\quad\text{and}\quad  \quad 3\,\lambda_{\theta\theta}=\kappa_{\theta\theta}\,,
\]
where subscripts denote the corresponding partial derivatives.
\end{lemma}
\begin{proof}
The claims can be checked using the Taylor expansion
\[
y(\tfrac\pi2+\omega)=\frac{1}{\kappa(\tfrac{\pi}{2})}\left(\o-\frac{1}{6}\left[1+\frac{\kappa_{\theta\theta}(\tfrac{\pi}{2})}{\kappa(\tfrac{\pi}{2})}\right]\o^3\right)+o(\o^4)\,,
\]
the fact that $\o\mapsto \kappa(\tfrac{\pi}{2}+\o)$ and $\o\mapsto \lambda(\tfrac{\pi}{2}+\o)$ are even functions, and the formula
\begin{equation}\label{eq:Dlambda}
\kappa\lambda_\theta=\frac{\sin\theta}{y}(\kappa-\lambda)
\end{equation}
%and
%\begin{equation}\label{eq:D2lambda}
%\kappa_\theta\lambda_\theta+\kappa\lambda_{\theta\theta}=(\kappa\lambda_\theta)_\theta=-\lambda(\kappa-\lambda)+\frac{\sin\theta}{y}\left(\kappa_\theta-2\lambda_\theta\right)
%\end{equation}
at points $\theta\in S^1\setminus\{\pm\tfrac{\pi}{2}\}$.
\end{proof}
%We have computed the derivatives of $\lambda$ at the poles as they will be needed in Lemma \ref{lem:criticalpoints}.

Note that the profile curve $\Gamma$ of an $O(n)$-invariant hypersurface $\Sigma^n$ is necessarily reflection symmetric in the $x$-axis (i.e. the line $\R e_1$). Conversely, if $\Gamma\hookrightarrow \E^2$ is a convex planar curve in $\E^2:=\spa\{e_1,e_2\}\subset \R^{n+1}$ which is reflection symmetric in the line $\R e_1$ then the hypersurface
\[
\Sigma^n:=\{x e_1+y\varphi: xe_1+ye_2\in \Gamma,\; \varphi\in S^{n-1}\subset \{e_1\}^\perp\}
\]
is smooth and $O(n)$-invariant with respect to $e_1$. It is clear that $\Sigma^n$ is smooth away from the poles $\theta=\pm\frac{\pi}{2}$. Smoothness of $\Sigma^n$ at the poles is readily verified by writing $\Sigma^n$ locally near $\theta=\pm\frac{\pi}{2}$ as a graph $(x,f(\vert x\vert))$, where $f$ is a smooth, even function.

\begin{comment}
\texttt{Lemma is probably overkill. Can remove.}
\begin{lemma}
Given a smooth, even function $f:(-1,1)\to \R$, define a function $u:B_1^n(0)\to \R$ by
\[
u(x):=f(|x|)\,.
\]
Then $u$ is smooth.
\end{lemma}
\begin{proof}
Since $f$ is even, its derivative at zero vanishes. It follows that the function $r\mapsto r^{-1}f'(r)$ is smooth and even. Iterating, we conclude that the function $\partial_\rho^k f$ is well-defined and smooth in $(-1,1)$ for all $k\in \N$, where the vector field $\partial_\rho$ is defined by $\partial_\rho:=r^{-1}\partial_r$. The claim follows, since the derivatives of $u$ are of the form
%\begin{equation*}
%D^{2k}u(x)=\sum_{i=0}^k\partial_\rho^{2k-i}f(|x|)\,x^{2(k-i)}\ast \mathrm{I}^i
%\end{equation*}
%and
%\begin{equation*}
%D^{2k+1}u(x)=\sum_{i=0}^k\partial_\rho^{2k+1-i}f(|x|)\,x^{2(k-i)+1}\ast \mathrm{I}^i\,,
%\end{equation*}
\begin{equation*}
D^{k}u(x)=\sum_{i=0}^{\left[\frac{k}{2}\right]}\partial_\rho^{k-i}f(|x|)\,x^{k-2i}\ast \mathrm{I}^i\,,
\end{equation*}
where $[\,\cdot \,]$ is the integer part of \,$\cdot$\, and $x^i\ast \mathrm{I}^j$ denotes, up to a normalization constant, the totally symmetric tensor product of $i$ factors of $x$ and $j$ factors of the identity $\mathrm{I}$.
\end{proof}
\end{comment}

\subsection{\texorpdfstring{$O(n)$}{\textit{O(n)}}-invariance and mean curvature flow} Next, given some convex, $O(n)$-invariant embedding $F_0:S^n\to \R^{n+1}$ let $F:S^n\times[-T,0)\to \R^{n+1}$ be the unique (convex, maximal) solution of mean curvature flow satisfying $F(\cdot,-T)=F_0$. %Without loss of generality, we can assume that $F_0$ converges to the origin as $t\to 0$ \cite{Hu84}. 
It follows from uniqueness of the solution and isometry invariance of the mean curvature flow that the $O(n)$-invariance of $F_0$ is preserved under the flow. Given $e_2\in \{e_1\}^\perp$, set $\E^2:=\spa\{e_1,e_2\}$ and denote by $\Gamma_t:=\Sigma^n_t\cap \E^2$ the corresponding profile curve for $\Sigma^n_t:=F(\Sigma^n,t)$. Of course, the normal speed of $\Gamma_t$ at a point $p$ is given by the mean curvature of $\Sigma^n_t$ at $p$. Thus, if we parametrize $\Gamma_t$ with respect to turning angle $\theta$ by a curve $\gamma:S^1\to \E^2$ then the component of the velocity normal to $\Gamma_t$ is given by
\[
(\partial_t\gamma)^\perp(\theta,t)=-H(\theta,t)\nu(\theta)\,.
\]
It is readily checked that the tangential component of the velocity must be $-H_\theta\tau$ (cf. \cite{GaHa86}) so that
\begin{equation}\label{eq:tangentialrotationalMCF}
\partial_t\gamma(\theta,t)=-H(\theta,t)\nu(\theta)-H_\theta(\theta,t)\tau(\theta)\,.
\end{equation}

%Conversely, given a solution of \eqref{eq:tangentialrotationalMCF}, we obtain an  $O(n)$-invariant solution of mean curvature flow in $\R^{n+1}$ by reparametrizing and rotating.

We will need the following evolution equations.
\begin{lemma}\label{eq:rotationalevolutionequations}
Let $\gamma:S^1\times[-T,0)\to\R^{n+1}$ be a solution of \eqref{eq:tangentialrotationalMCF}, where $H$ is given by \eqref{eq:rotationalH}. Then
\begin{align}
\kappa_t%={}&\kappa^2\left(H_{\theta\theta}+H\right)\label{eq:rotationalevolutionkappa1}\\
={}&\kappa^2\kappa_{\theta\theta}+\vert \mathrm{II}\vert^2\kappa-(n-1)\lambda\tan\theta\left(\lambda\kappa_\theta-2\kappa\lambda_\theta\right)\,,\label{eq:rotationalevolutionkappa2}
\end{align}
\begin{equation}\label{eq:rotationalevolutionlambda}
\lambda_t=\kappa^2\lambda_{\theta\theta}+\vert \mathrm{II}\vert^2\lambda-\lambda(H+\kappa)\tan\theta \lambda_\theta\,,
\end{equation}
\begin{equation}\label{eq:rotationalevolutionH}
H_t=\kappa^2H_{\theta\theta}+\vert \mathrm{II}\vert^2H-(n-1)\lambda^2\tan\theta H_\theta
\end{equation}
and
\begin{equation}\label{eq:rotationalevolutionarea}
-\frac{d}{dt}A=2\pi+(n-1)\int\frac{\lambda}{\kappa}\,d\theta\,,
\end{equation}
where $\vert \mathrm{II}\vert^2:=\kappa^2+(n-1)\lambda^2$ and $A(t)$ is the area enclosed by $\Gamma_t$.
\end{lemma}
\begin{proof}
First, we compute directly from \eqref{eq:tangentialrotationalMCF} (cf. \cite{GaHa86})
\begin{equation}
\kappa_t=\kappa^2\left(H_{\theta\theta}+H\right)\label{eq:rotationalevolutionkappa1}
\end{equation}
and, from the definition of $\lambda$,
\[
\lambda_t=\lambda^2\left(H-\tan\theta H_\theta\right)\,.
\]
The first two identities, \eqref{eq:rotationalevolutionkappa2} and \eqref{eq:rotationalevolutionlambda}, now follow from the identity \eqref{eq:Dlambda} and its derivative
\begin{equation}\label{eq:D2lambda}
\kappa_\theta\lambda_\theta+\kappa\lambda_{\theta\theta}=(\kappa\lambda_\theta)_\theta=-\lambda(\kappa-\lambda)+\frac{\sin\theta}{y}\left(\kappa_\theta-2\lambda_\theta\right)\,.
\end{equation}
Together they imply the identity \eqref{eq:rotationalevolutionH} for $H$. To obtain the final identity, we observe that
\[
-\frac{d}{dt}A=\int_0^{L}H\,ds=2\pi+(n-1)\int\frac{\lambda}{\kappa}\,d\theta\,,
\]
where $s(t)$ is the arc-length parameter for $\Gamma_t$ and $L(t)$ is its length.
\end{proof}

\subsection{The approximating solutions}\label{sec:approximatingsolutions}

Our approximating solutions are given by evolving rotated timeslices of the Angenent oval by mean curvature flow: Denote the Angenent oval by $\gamma:S^1\times(-\infty,0)\to\R^2$ and set $\Gamma_t:=\gamma(S^1,t)$. Given any $R>0$, we rotate the curve $\Gamma^R:=\Gamma_{-R}$ about the $x$-axis (the `short' axis) to form an $O(n)$-invariant hypersurface $\Sigma^R$ in $\R^{n+1}$; that is, identifying $\R^2$ with $\E^2:=\spa\{e_1,e_2\}\subset \R^{n+1}$, we consider the hypersurface
\[
\Sigma^R:=\{x_R(\theta)e_1+y_R(\theta)\varphi:\theta\in S^1,\; \varphi\in S^{n-1}\subset \{e_1\}^\perp\}\,,
\]
where $x_R$ and $y_R$ are defined by $\gamma(\theta,-R)=x_R(\theta)e_1+y_R(\theta)e_2$. 

We want to evolve the rotated ovals by mean curvature flow. So fix an initial parametrization $F^R_0:\Sigma^R\to \R^{n+1}$ and consider the $O(n)$-invariant solution $F_R:S^n\times[-T_R,0)\to\R^{n+1}$ of mean curvature flow with initial datum $F_R(\cdot,0)=F^R_0$.

Finally, we denote by $\Gamma^R_t:=\Sigma^R_t\cap \E^2$ be the profile curve of $\Sigma^R_t:=F_R(S^n,t)$ and denote its turning angle parametrization by $\gamma_R:S^1\times[-T_R,0)$. Then $\gamma_R$ satisfies \eqref{eq:tangentialrotationalMCF}. As usual, we normalize $\theta$ so that the unit tangent vector field $\tau_R$ satisfies $\tau_R(0)=e_1$.

By uniqueness of solutions of the mean curvature flow, we note that the reflection symmetry of $\gamma_R$ in both axes is preserved under the flow. Since the solution also remains strictly convex \cite{Hu84}, it follows that the points $\gamma_R(\tfrac{\pi}{2},t)$ and $\gamma_R(\pi,t)$ are the unique points of $\Gamma^R_t$ which lie on the positive $x$-axis and the positive $y$-axis respectively. These points will play an important role in our analysis.

\section{Existence}

We want to show that our sequence of approximating solutions $F_R:S^n\times[-T_R,0)\to\R^{n+1}$ converges to a compact ancient solution of mean curvature flow along some subsequence $R_i\to\infty$. It will suffice to obtain uniform estimates for the time of existence, $T_R$, the mean curvature, $H_R$, and for the vertical displacement. We will also need a lower bound for the horizontal displacement to ensure that the limit solution does not lie in a smaller slab. We will return to the approximating solutions in Section \ref{sec:uniqueness}, where they play an important role in our investigation of the asymptotics of the limit solution.

Observe that, by the reflection symmetries of $\gamma_R$, $\partial_\theta\kappa_R(\theta,t)=0$ when $\theta\in\{0,\tfrac{\pi}{2},\pi,\tfrac{3\pi}{2}\}$ and, by the rotation symmetry of $\gamma_R$, $(\kappa_R-\lambda_R)(\theta,t)=0$ when $\theta\in\{\tfrac{\pi}{2},\tfrac{3\pi}{2}\}$ for all $t\in [-T_R,0)$. Recalling the formula \eqref{eq:Dlambda} for $\lambda_\theta$, this implies, in particular, that the set $\{0,\tfrac{\pi}{2},\pi,\tfrac{3\pi}{2}\}$ is critical for $H_R$ for all $t\in[-T_R,0)$. We will see that these are the only critical points of $\kappa_R$ and $\lambda_R$ (and hence $H_R$) at the initial time. Using the maximum principle, we will show that $\kappa_R$ and $\lambda_R$ develops no new critical points along $\gamma_R$ during its evolution.

\begin{lemma} \label{lem:criticalpoints}
For every $t\in [-T_R,0)$
\[
\partial_\theta\k_R(\cdot,t)>0\;\;\text{in}\;\;\left(\tfrac{\pi}{2},\pi \right)\;\;\;\text{and}\;\;\; (\kappa_R-\lambda_R)(\cdot,t)>0\;\;\text{in}\;\;\left(\tfrac{\pi}{2},\tfrac{3\pi}{2}\right)\,.
\]
In particular, for every $t\in[-T_R,0)$,
\[
H^R_{\min}(t):=\min_{\omega\in S^1}H_R(\omega,t)=H_R\left(\tfrac{\pi}{2},t\right)<H_R(\theta,t)\quad\text{for all}\quad \theta\in \left(\tfrac{\pi}{2},\pi\right]
\]
and
\[
H^R_{\max}(t):=\max_{\omega\in S^1}H_R(\omega,t)=H_R(\pi,t)>H_R(\theta,t)\quad\text{for all}\quad \theta\in \left[\tfrac{\pi}{2},\pi\right)\,.
\]
\end{lemma}
\begin{proof}
%We will show that $\kappa_R-\lambda_R$ and $\partial_\theta\kappa_R$ are positive in $\left(-\frac{\pi}{2},\frac{\pi}{2}\right)\times[-T_R,0)$ and $\left(-\frac{\pi}{2},0\right)\times[-T_R,0)$ respectively. By the formula \eqref{eq:Dlambda} for $\lambda_\theta$ and the symmetries of $\gamma_R$, this implies the claims. 

For convenience, we will drop the subscript $R$ and use subscripts $t$ and $\theta$ to denote the corresponding partial derivatives.

Consider the functions $u:=\frac{\kappa}{\lambda}$ and $v:=\kappa_\theta$. By Lemma \ref{lem:lambdalemma}, $u=1$ on the spatial boundary of $(-\frac{\pi}{2},\frac{\pi}{2})\times[-T,0)$. We claim that $u>1$ in $(-\frac{\pi}{2},\frac{\pi}{2})$ at the initial time. Indeed, differentiating \eqref{k} we find that $v=\kappa_\theta>0$ when $\theta\in(\frac{\pi}{2},\pi)$. Thus,
\[
y(\theta,-T)=\int_{\frac{\pi}{2}}^\theta\frac{\sin u}{\kappa(u,-T)}du>-\frac{\cos\theta}{\kappa(\theta,-T)}
\]
and hence, by the definition of $\lambda$, $\kappa(\theta,-T)>\lambda(\theta,-T)$ for any $\theta\in (\frac{\pi}{2},\pi)$. The claim then follows from the symmetries of the profile curve.

Recalling the evolution equations \eqref{eq:rotationalevolutionkappa2} and \eqref{eq:rotationalevolutionlambda} for $\kappa$ and $\lambda$ from Lemma \ref{eq:rotationalevolutionequations}, we find that $u$ satisfies
\begin{equation}\label{eq:rotationalevolutionu}
u_t-\kappa^2u_{\theta\theta}-2\kappa^2u_\theta\frac{\lambda_\theta}{\lambda}=-(n-1)\lambda^2\tan\theta\, u_\theta-2\tan^2\theta H(\kappa-\lambda)
\end{equation}
in $(-\frac{\pi}{2},\frac{\pi}{2})\times[-T,0)$. Thus, if $u$ reaches a new local minimum at some interior point then, at that point,
\[
0\geq u_t-\kappa^2u_{\theta\theta}=-2\tan^2\theta H(\kappa-\lambda)
\]
and hence $\kappa\geq \lambda$. We conclude that $u\geq 1$ in $[-\frac{\pi}{2},\frac{\pi}{2}]\times[-T,0)$ and hence in all of $S^1\times[-T,0)$. In fact, by the strong maximum principle, we have the strict inequality $u>1$ in the interior $(-\frac{\pi}{2},\frac{\pi}{2})\times(-T,0)$.

Next, we consider the function $v:=\kappa_\theta$ for $\theta\in(\frac{\pi}{2},\pi)$. By the reflection symmetries of the profile curve, we find that $v=0$ on the spatial boundary $\{\frac{\pi}{2},\pi\}\times[-T,0)$ and we saw above that $v>0$ in $(\frac{\pi}{2},\pi)$ at the initial time. Recalling \eqref{eq:rotationalevolutionkappa1}, we find that $v$ satisfies
\begin{align*}
v_t={}&\kappa^2(H_{\theta\theta}+H)_\theta+2\kappa\kappa_\theta(H_{\theta\theta}+H)\\
={}&\kappa^2(v_{\theta\theta}+v)+2\kappa v(v_\theta+\kappa)+(n-1)\kappa\left[2v(\lambda_{\theta\theta}+\lambda)+\kappa(\lambda_{\theta\theta}+\lambda)_\theta\right]
\end{align*}
in $(\frac{\pi}{2},\pi)\times (-T,0)$. 

To control the final term, we differentiate \eqref{eq:D2lambda} to obtain
%\begin{equation*}%\label{eq:D2lambda}
%\kappa_\theta\lambda_\theta+\kappa\lambda_{\theta\theta}=(\kappa\lambda_\theta)_\theta=-\lambda(\kappa-\lambda)+\frac{\sin\theta}{y}\left(\kappa_\theta-2\lambda_\theta\right)
%\end{equation*}
%and, differentiating once more,
\begin{align*}
\kappa_{\theta\theta}\lambda_\theta+2\kappa_\theta\lambda_{\theta\theta} +\kappa\lambda_{\theta\theta\theta}={}&2\lambda\lambda_\theta-\kappa\lambda_\theta -\lambda\kappa_\theta+\frac{\kappa\lambda_\theta}{\kappa-\lambda}(\kappa_{\theta\theta}-2\lambda_{\theta\theta})\\
{}&-\left(\lambda+\frac{\kappa\lambda^2_\theta}{(\kappa-\lambda)^2}\right)(\kappa_\theta-2\lambda_\theta)\,.
\end{align*}
Combining this with \eqref{eq:D2lambda}, we find
\begin{align*}
\kappa(\lambda_{\theta\theta}+\lambda)_\theta+2v(\lambda_{\theta\theta}+\lambda)={}&6\lambda_\theta\left(\lambda+\frac{\kappa\lambda_\theta^2}{(\kappa-\lambda)^2}\right)+\frac{\lambda\lambda_\theta}{\kappa-\lambda}v_{\theta}\\
{}&-\frac{\lambda_\theta^2}{(\kappa-\lambda)^2}(H+\lambda)v\,.
\end{align*}
We have proved that $\kappa>\lambda$, so we can estimate $H+\lambda< 3\kappa$ and, recalling \eqref{eq:Dlambda}, $\lambda_\theta>0$. Thus,
\begin{align*}
v_t\geq{}&\kappa^2(v_{\theta\theta}+v)+2\kappa v(v_\theta+\kappa)+(n-1)\kappa\left(\frac{\lambda\lambda_\theta}{\kappa-\lambda}v_{\theta}-3\frac{\lambda_\theta^2}{(\kappa-\lambda)^2}\kappa v\right)\\
={}&\kappa^2v_{\theta\theta}+(2\kappa v-(n-1)\lambda^2\tan\theta)v_\theta+3(\kappa^2-(n-1)\lambda^2\tan^2\theta)v\,.
\end{align*}
The strong maximum principle now implies that $v>0$ in $(\frac{\pi}{2},\pi)\times(-T,0)$, completing the proof of the lemma.
\end{proof}

For each $t\in [-T_R,0)$ we define the horizontal and vertical displacements
\begin{equation}\label{distances}
h_R(t):=\max_{\theta\in S^1}|\langle \gamma_R(\theta, t),e_1\rangle| \quad \text{and}\quad \ell_R(t):=\max_{\theta\in S^1}|\langle \gamma_R(\theta, t), e_2\rangle|\,.
\end{equation}
Note that, since $\Gamma^R_t$ is convex and symmetric under reflection about the axes,
\[
h_R(t)=\langle\gamma_R(\tfrac{\pi}{2},t),e_1\rangle\quad \text{and}\quad \ell_R(t)=\langle\gamma_R(\pi,t),e_2\rangle\,.
\]

The product of $h_R$ and $\ell_R$ is controlled by the area enclosed by $\gamma_R$.
\begin{lemma}\label{hl-lemma} For any $t\in [-T_R,0)$,
\[
-\tfrac{\pi}{2}t\le h_R(t)\ell_R(t)\le -n\pi t\,,
\]
where $h_R(t)$ and $\ell_R(t)$ are given by \eqref{distances}.
\end{lemma}
\begin{proof}
Let $A_R(t)$ denote the area of the region enclosed by the profile curve. Then, estimating $0<\lambda\leq\kappa$ and integrating \eqref{eq:rotationalevolutionarea}, we may estimate
\begin{equation}\label{area-t estimate}
-2\pi t\le A_R(t)\le -2n\pi t\,.
\end{equation}
On the other hand, since $\Gamma^R_t$ is convex, it lies inside the rectangle $[-h_R(t),h_R(t)]\times[-\ell_R(t),\ell_R(t)]$ and outside the parallelogram with vertices $\pm\gamma_R(\tfrac{\pi}{2},t)=\pm(h_R(t),0)$ and $\pm\gamma_R(\pi,t)=\pm(0,\ell_R(t))$, and hence
\begin{equation}\label{area-hl estimate}
2 h_R(t)\ell_R(t)\le A_R(t)\le 4h_R(t)\ell_R(t).
\end{equation}
Combining \eqref{area-t estimate} and \eqref{area-hl estimate} yields the lemma.
\end{proof}
Combining Lemma \ref{hl-lemma} with the estimates, provided by Lemma \ref{CR-lem}, for the initial values $h_R(-T_R)$ and $\ell_R(-T_R)$ yields a crude estimate for $T_R$. %(Subsequently, by estimating $\lambda_R$ more carefully, we will derive a more precise relation between $T_R$ and $R$).
\begin{corollary} \label{time}
\[
\begin{split}
\frac{R}{2n} (1-\mathrm{e}^{-R})\le T_R\le R+\log 2\,.
\end{split}
\]
\end{corollary}

Next, we bound the minimum of the mean curvature in terms of the horizontal and vertical displacements.
\begin{lemma} \label{Hmin-lemma}
For every $t\in [-T_R,0)$,
\[
H^R_{\min}(t)\le\frac{2nh_R(t)}{\ell_R^2(t)+h_R^2(t)}\,.
\]
%where $h_R(t)$ and $\ell_R(t)$ are given by \eqref{distances}.
\end{lemma}
\begin{proof}\mbox{}
Fix any $t\in [-T_R,0)$ and write $h_R=h_R(t)$ and $\ell_R=\ell_R(t)$. Consider, for some $\rho>h_R$, the point $A:=(-(\rho-h_R), 0)$ and the circle $\mathcal C$ centered at $A$ and with radius $\rho$. Then $\gamma_R(\pi,t)=(0,\ell_R)$ lies outside $\mathcal C$ provided that (see Figure \ref{fig3})
\begin{equation}\label{rho}
(\rho-h_R)^2+\ell_R^2>\rho^2\quad\Longleftrightarrow\quad \rho<\frac{\ell_R^2+h_R^2}{2h_R}\,.
\end{equation}
Note that for all $t\in [-T_R,0)$ $\ell_R(t)> h_R(t)$, since, by Lemma \ref{lem:criticalpoints}, $H_R(\tfrac\pi2,t)>H_R(\pi,t)$.
\begin{figure}
\begin{center}
\includegraphics[width=0.7\textwidth]{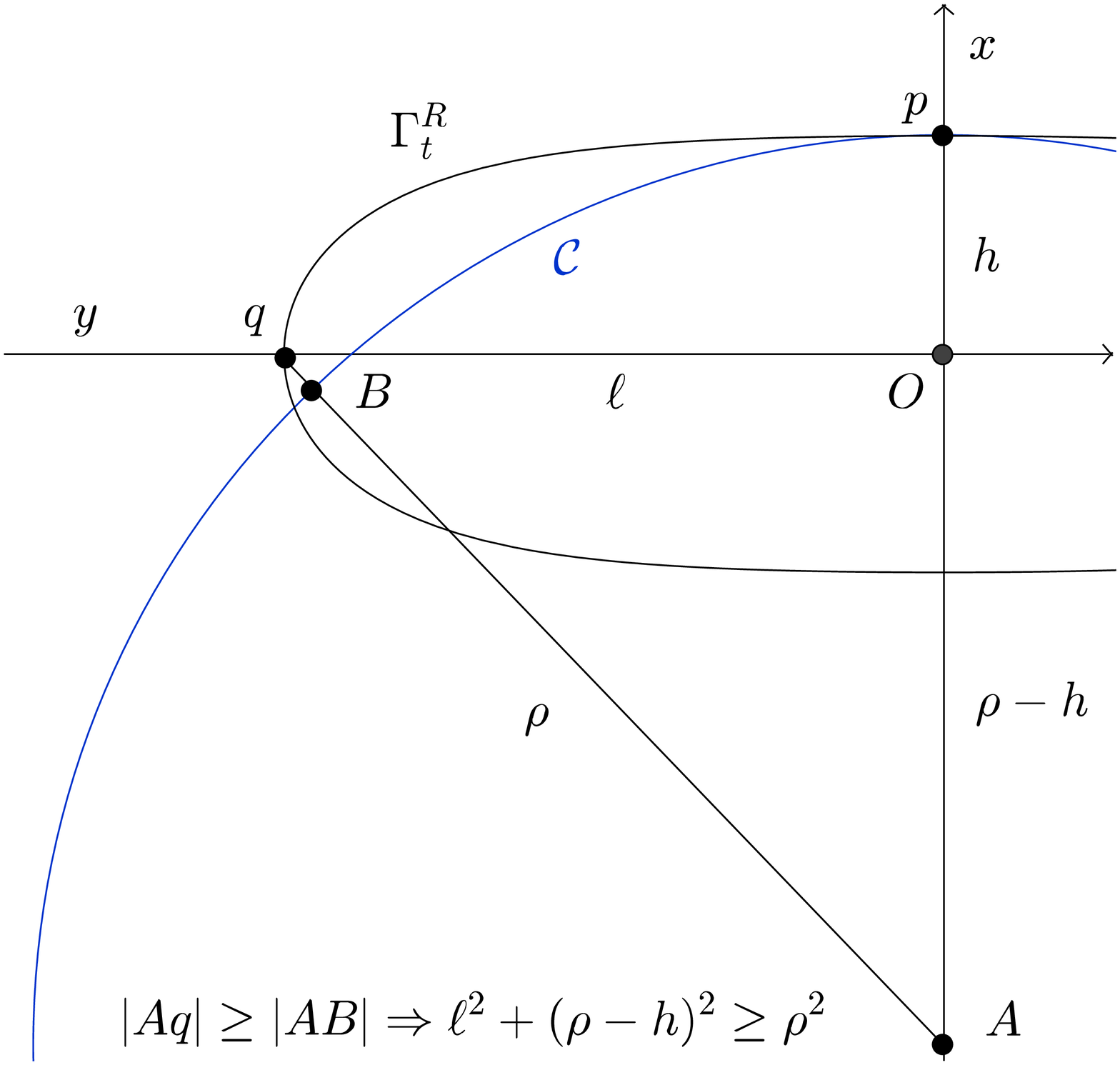}
\end{center}
\caption{Bounding $H_{\mathrm{min}}$.}\label{fig3}
\end{figure}

Consider now  any $\rho>h_R$ satisfying \eqref{rho}. Note that $\gamma_R(\tfrac{\pi}{2},t)=(h_R,0)\in \mathcal C$ and, by the reflection symmetry of $\Gamma^R_t$ along the $x$-axis, $\Gamma^R_t$ and $\mathcal C$ must be tangent at $\gamma_R(\tfrac{\pi}{2},t)=(h_R,0)$. 
\begin{claim}\label{circle-claim}
$\Gamma^R_t$ lies outside of $\mathcal C$ near the point $\gamma_R(\tfrac{\pi}{2},t)=(h_R,0)$. 
\end{claim}
Once we have established Claim \ref{circle-claim} the lemma follows. Indeed, by the rotational symmetry of $\Gamma^R_t$ about the $x$-axis, $\Sigma^R_t$ must lie outside the sphere centered at $A$ of radius $\rho$ around the point $(h_R,0,0)$ and hence, by the maximum principle,
\[
H_R(\tfrac{\pi}{2},t)\le \frac{n}{\rho}\quad \text{for any}\quad\rho<\frac{\ell_R^2+h_R^2}{2h_R}\,.
\]
\begin{proof}[Proof of Claim \ref{circle-claim}]
Suppose the Claim is false. Then, by its reflection symmetry, $\Gamma^R_t$ lies to the \emph{inside} of $\mathcal C$ locally around $(h_R,0)$. By the rotational symmetry of $\Sigma^R_t$ and the maximum principle, we conclude $H_R(\tfrac{\pi}{2},t)\ge 2 \rho^{-1}$. On the other hand, since the point $\gamma_R(\pi,t)$ lies outside of $\mathcal C$, there must exist a point $\theta_0\in(\tfrac{\pi}{2},\pi)$ such that $\gamma_R(\theta_0,t)\in \mathcal C\cap\Gamma^R_t$ and $\gamma_R(\theta,t)$ lies inside $\mathcal C$ for all $\theta\in (\tfrac{\pi}{2},\theta_0)$. We consider now the translations $\mathcal C_\l:=\mathcal C-\l e_1$ of $\mathcal C$ along the $-e_1$ direction. By the preceding arguments, there is certainly some $\l_0>0$ and some $\theta_1\in(\tfrac{\pi}{2},\theta_0)$ such that $\mathcal{C}_{\l_0}$ is tangent to $\Gamma^R_t$ at $\gamma_R(\theta_1,t)$. The maximum principle for the curvature then implies that
\[
\k_R(\theta_1,t)\le \frac1\rho.
\]
Furthermore
\[
\nu_R(\theta_1,t) =\frac{\gamma_R(\theta_1,t)-(A-\l_0 e_1)}{|\gamma_R(\theta_1,t)-(A-\l_0 e_1)|}=\frac{\gamma_R(\theta_1,t)-(A-\l_0 e_1)}{\rho}
\]
and thus
\[
\frac{\langle\nu_R(\theta_1,t),e_2\rangle}{\langle\gamma_R(\theta_1,t),e_2\rangle}=\frac1\rho.
\]
It follows that $H_R(\theta_1,t)\le \frac{n}{\rho}\le H_R(\tfrac{\pi}{2},t)$. Lemma \ref{lem:criticalpoints} now implies that $\theta_1=\tfrac{\pi}{2}$, a contradiction.
\end{proof}
As explained above, this completes the proof of the lemma.
\end{proof}

Lemmas \ref{lem:criticalpoints}, \ref{hl-lemma} and \ref{Hmin-lemma} can now be combined to bound the horizontal displacement $h_R$ from below.
\begin{lemma} \label{h-lemma}
For every $t\in [-T_R,0)$
\[
h_R(t)\ge\frac\pi2(1- \mathrm{e}^{-R}) \mathrm{e}^{\frac{2n}{t}}\,.
\]
\end{lemma}

\begin{proof} Using Lemmas \ref{lem:criticalpoints} and \ref{Hmin-lemma}, we find
\begin{align*}
\frac{d(-h_R(t))}{dt}= H_R(\tfrac{\pi}{2},t)=H^R_{\min}(t)\le{}& \frac{2nh_R(t)}{\ell_R^2(t)+h_R^2(t)}\le\frac{2nh_R(t)}{\ell^2_R(t)}\,.
\end{align*}
On the other hand, Lemma \ref{hl-lemma} yields the estimate
\[
\ell_R(t)\ge \frac{-\pi t}{2h_R(t)}\ge -t
\]
so that
\[
\frac{d}{dt}(-\log (h_R(t)))\le\frac{2n}{t^2}=\frac{d}{dt}\left(\frac{2n}{-t}\right)\,.
\]
Integrating yields
\begin{align*}
-\log(h_R(t))+\log(h_R(-T_R))\le\frac{2n}{-t}-\frac{2n}{T_R}.
\end{align*}
Using Lemma \ref{CR-lem} to bound $h_R(-T_R)=h(-R)$  we obtain
\[
 h_R(t)\ge h_R(-T_R)\exp\left(\frac{2n}{T_R}+\frac{2n}{t}\right)\ge \frac\pi2(1- \mathrm{e}^{-R}) \mathrm{e}^{\frac{2n}{t}}\,.
\]
\end{proof}

Combining Lemmas \ref{hl-lemma} and \ref{h-lemma} yields an estimate for $\ell_R(t)$.
\begin{corollary}\label{cor:lbound}
For every $t\in[-T_R, 0)$,
\[
\ell_R(t)\le -\frac{2n\mathrm{e}^{-\frac{2n}{t}}}{1- \mathrm{e}^{-R}}t\,.
\]
\end{corollary}

Using the Harnack inequality, we can now estimate $H_{\max}$.
\begin{lemma}\label{lem:Hbound}
For any $t\ge-\frac {T_R}{2}$,
\[
H^R_{\max}(t)=H_R(\pi,t)\le\frac{2n\sqrt 2\mathrm{e}^{-\frac{2n}{t}}}{1- \mathrm{e}^{-R}}\,.
\]
\end{lemma}
\begin{proof}
By the differential Harnack estimate \cite[Corollary 1.2]{Ham95},
\begin{equation}\label{HI}
H_R(\pi,s)\ge \sqrt{\frac{t+T_R}{s+T_R}}H_R(\pi,t)
\end{equation}
for $-T_R\le t<s<0$. Lemma \ref{lem:criticalpoints} then yields
\[
\begin{split}
\ell_R(t)&=\int_t^0 H_R(\pi,s)ds\ge H_R(\pi,t)\sqrt{t+T_R}\int_t^0\frac{1}{\sqrt{s+T_R}}ds\\
&=H^R_{\max}(t)2\sqrt{t+T_R}\left(\sqrt T_R-\sqrt{T_R+t}\right)\\
&=H^R_{\max}(t)2\sqrt{t+T_R}\frac{-t}{\sqrt T_R+\sqrt{T_R+t}}\\
&\ge H^R_{\max}(t)\sqrt{t+T_R}\frac{-t}{\sqrt T_R}\,.
\end{split}
\]
For $t>-\frac{T_R}{2}$, we obtain
\[
\begin{split}
\ell_R(t)\ge H^R_{\max}(t)\frac{-t}{\sqrt 2}\,.
\end{split}
\]
The claim now follows from Corollary \ref{cor:lbound}.
\end{proof}

Any of various well-known compactness arguments now yield the desired ancient solution (see for example %\cite{Cooper} or 
\cite[Chapters 3 and 4]{Mantegazza}).
%{\color{red} Maybe we should change the following theorem to: ``the $F_R$'s have a limit which is an ancient solution'' Maybe also add "later we provide precise asymptotics blah blah blah". Here we should also cite again Wang. Note Wang's result does not include rotationally symmetric. He shows that one can construct NOT $k$-rotationally symmetric for any $k$.}
\begin{thm}\label{thm:limit}
There is a sequence of approximating solutions which converges locally uniformly in the smooth topology to a smooth, compact, convex, $O(1)\times O(n)$-invariant ancient solution of mean curvature flow which lies in the slab $\Omega:=\{(x,y,z)\in \R\times\R\times\R^{n-1}:|x|<\frac{\pi}{2}\}$ and in no smaller slab. 
%Let $\{R_i\}_{i\in\N}$ be a sequence of positive numbers which goes to infinity as $i$ goes to infinity, and denote by $F_i:S^2\times[-T_i,0)\to\R^n$ the solution of \eqref{SR-MCF} with $R=R_i$, where $T_i:=T_{R_i}$. The sequence $\{F_i\}_{i\in \N}$ converges locally uniformly in $C^\infty(S^2\times(-\infty,0);\R^{3})$ to a convex ancient limit solution $F_\infty:S^2\times (-\infty,0)\to \R^{3}$ of mean curvature flow which lies in the slab $\Omega:=\{(x,y,z)\in \R^3:-\pi/2\leq x\leq \pi/2\}$ and in no smaller slab. Moreover, the limit solution $F_\infty$ is unique: If $\{R'_i\}_{i\in\N}$ is another sequence of positive numbers which goes to infinity as $i$ goes to infinity, then the corresponding sequence $\{F'_{i}\}_{i\in \N}$ of solutions of \eqref{SR-MCF} converges locally uniformly in $C^\infty(S^2\times(-\infty,0);\R^{3})$ to $F_\infty$, modulo a diffeomorphism of $S^2$.
\end{thm}
\begin{proof}
Let $\{R_i\}_{i\in\N}$ be a sequence of positive numbers which go to infinity as $i$ goes to infinity and denote by $F_i:=F_{R_i}:S^n\times[-T_i,0)\to\R^{n+1}$ the corresponding approximating solution, where $T_i:=T_{R_i}$. %We will show that a subsequence of the solutions $F_i$ converges locally uniformly in $C^\infty(S^n\times(-\infty,0);\R^{n+1})$ to a limit solution $F_\infty:S^n\times (-\infty,0)\to \R^{n+1}$ of mean curvature flow which lies in the slab $\Omega:=\{(x,y,z)\in \R^3:-\tfrac{\pi}{2}\leq x\leq \tfrac{\pi}{2}\}$ (and in no smaller slab).
By Corollary \ref{time}, $T_i$ goes to infinity as $i$ goes to infinity. Since any convex hypersurface satisfies $|\mathrm{II}|^2\leq H^2$, Lemma \ref{lem:Hbound} yields a bound for the curvature $|\mathrm{II}_i|$ of $F_i$ on any compact subset $K\subset S^n\times (-\infty,0)$ uniform in $i$. The estimates of \cite{EcHu91} then provide, for every $k\in\N$ and $m\in \N$, bounds for $|\nabla^k_t\nabla^m\mathrm{II}_i|$ on $K$, uniformly in $i$. Combined with a uniform bound for the displacement $|F_i|$ coming from Corollary \ref{cor:lbound}, these estimates can be translated into uniform bounds for the derivatives $|\partial^k_t\partial^m_\theta\partial^n_{\varphi^j} F_i|$ on $K$. By the Arzel\`a--Ascoli Theorem and a diagonal subsequence argument, a subsequence of the flows $F_i$ converge locally uniformly in $C^\infty$ to an ancient solution $F_\infty:S^n\times(-\infty,0)\to\R^{n+1}$. The limit is certainly weakly convex, $\mathrm{II}_\infty\geq 0$, since this is the case for each $F_i$. Strict convexity then follows from the strong maximum principle for the second fundamental form and compactness of the limit. The limit is $O(1)\times O(n)$-invariant and lies in the slab $\Omega$ for all $t\in(-\infty,0)$ since this is the case for each of the approximating solutions. By Lemma \ref{h-lemma}, the limit cannot lie in any smaller slab. %It remains to show that the limit is unique. Let $\{R'_i\}_{i\in \N}$ be any other sequence of positive numbers converging to $\infty$ and denote by $F'_i$ the corresponding solution of \eqref{SR-MCF}. By the preceding arguments, a subsequence of $\{F'_i\}_{i\in \N}$ converges locally uniformly in $C^\infty(S^2\times(-\infty,0);\R^{3})$ to some limit $F'_\infty$. Denote by $\Omega_{i,t}$ and $\Omega'_{i,t}$ the convex bodies bounded by $F_i(S^2,t)$ and $F'_i(S^2,t)$ respectively. For each $i\in \N$, we may choose some $i'\in \N$ such that $R_i\leq R'_{i'}$ and hence, by the avoidance principle, $\Omega_{i,t}\subset \Omega'_{i',T_i-T'_{i'}+t}$ for all $t\in[-T_i,0)$...
\end{proof}

\section{Unique asymptotics}\label{sec:uniqueasymptotics}

We now want to study $O(n)$-invariant ancient solutions lying in a slab more generally. Our ultimate goal is to show that the solution constructed in Theorem \ref{thm:limit} is the unique such solution. As a first step, we will show that all such solutions have the correct asymptotics. Indeed, well-known arguments show that the `parabolic' region is asymptotic to the boundary of the slab (Lemma \ref{lem:asymptoticsmiddle}) while the `edge' region is asymptotic to a Grim hyperplane (Lemma \ref{lem:grimconvergence}). By carefully estimating the area enclosed by the profile curve (Lemma \ref{lem:areaupperbound1}), we are able to show that this Grim hyperplane must have the same width as the slab (Corollary \ref{cor:asymptoticsedge}). This is quite a strong conclusion as it provides an asymptotic description of the solution all the way up to the parabolic region. This makes the treatment of the `intermediate' region in our case quite different from the one required for the ancient solutions constructed in \cite{HaHe} (see \cite{ADS}). %Combined with the Alexandrov reflection principle, it implies that the solution is reflection symmetric in the `$y$-axis' Lemma \ref{}). %Coupled with the displacement estimate of the following section, this will allow us to obtain the desired uniqueness result in Section \ref{sec:uniqueness}.

So consider any convex ancient solution $F:S^n\times(-\infty,0)\to\R^{n+1}$ of mean curvature flow that is $O(n)$-invariant with respect to some $e_1\in S^n$ and lies in the slab $\Omega:=\{(x,y,z)\in \R\times\R\times\R^{n-1}:\vert x\vert<\tfrac{\pi}{2}\}$ (and in no smaller slab). We write $\Sigma_t=\Sigma^n_t:=F(S^n,t)$ and, given some $e_2\in  S^{n-1}\subset \{e_1\}^\perp$, parametrize the profile curve $\Gamma_t:=\Sigma^n_t\cap \E^2$ with respect to turning angle by a curve $\gamma:S^1\times(-\infty,0)\to \E^2$, where $\E^2:=\spa\{e_1,e_2\}\cong \{(x,y,z)\in \R\times\R\times\R^{n-2}:z=0\}$. As usual, we denote by $\tau$ and $\nu$ the unit tangent vector and outward pointing unit normal and by $\k(\cdot, t)$ and $\l(\cdot,t)$ the curvature and rotational curvature, so that $H(\cdot, t)=\k(\cdot, t)+(n-1)\l(\cdot, t)$ is the mean curvature of $F$.

Finally, for each $t<0$ we define the vertical displacement
\[
\ell(t):=\max_{\theta\in S^1}|\la\gamma(\theta, t), e_2\ra|\,.
\]
Since $\Gamma_t$ is convex and rotationally symmetric,
\[
\ell (t)=\la\gamma(\pi,t), e_2\ra=-\la\gamma(0, t), e_2\ra\,.
\] 
We note that, unlike the situation in the previous section, the curve $\Gamma_t$ is not, a priori, reflection symmetric in the $y$-axis.

We first show that the parabolic region looks like the boundary of the slab as $t\to-\infty$.
\begin{lemma}\label{lem:asymptoticsmiddle}
For any sequence of times $t_i\to-\infty$ the sequence of mean curvature flows $\Sigma_t^i$ defined for $t\in (-\infty,-t_i)$ by $\Sigma^i_t:=\Sigma_{t+t_i}$ converges locally uniformly in the smooth topology to the stationary mean curvature flow defined by $\partial \Omega$.
\end{lemma}
\begin{proof}
Fix any sequence of times $t_i\to -\infty$ and any $t<0$. Since $\Omega$ is the smallest slab containing the solution, there exists, by the avoidance principle, sequences of points $\overline x_i$ and $\underline x_i\in \Sigma_{t+t_i}$ satisfying $\left\langle \overline x_i,e_1\right\rangle\to \frac{\pi}{2}$ and $\left\langle \underline x_i,e_1\right\rangle\to -\frac{\pi}{2}$ as $i\to\infty$. On the other hand, by the $O(n)$-invariance of the solution and the avoidance principle, there exists a sequence of points $y_i\in \Sigma_{t+t_i}$ such that $\left\langle y_i,e_2\right\rangle\to \infty$. By convexity and $O(n)$-invariance of $\Sigma_t$, we conclude that $\Sigma^i_{t}$ converges locally uniformly in the Hausdorff topology to $\partial \Omega$ as $i\to \infty$. That the convergence is in the smooth topology follows as in Theorem \ref{thm:limit} since, by the Harnack inequality, the mean curvature is monotone non-decreasing in $t$.
\end{proof}

We next show that the `edge' region looks like a Grim hyperplane of width $\alpha\pi$ as $t\to-\infty$, where $\alpha\in (0,1]$ is defined by
\begin{equation}\label{eq:alpha}
\alpha^{-1}:=\lim_{t\to-\infty} H(\pi,t)\,.
\end{equation}
Note that the limit exists since, by the Harnack inequality, $H$ is non-decreasing in $t$. We denote the inverse of the Gauss map of $\Sigma_t$ by $P$; that is, given $v\in S^{n}$, $P(v,t)$ is the point of  $\Sigma_t$ with unit outward normal $v$.

\begin{lemma}\label{lem:grimconvergence}
For any sequence of times $t_i\to-\infty$ and any unit vector $\varphi\in \{e_1\}^\perp$ the sequence of mean curvature flows $\{\Sigma^i_t\}_{t\in (-\infty,-t_i)}$ defined by $\Sigma^i_t:=\Sigma_{t+t_i}-P(\varphi,t_i)$ converges locally uniformly in the smooth topology to the scaled Grim hyperplane $\alpha \mathrm{G}^n_{\alpha^{-2}t}$, where $\mathrm{G}_t$ is the Grim hyperplane which translates in the direction $\varphi$ with unit speed; i.e.
\begin{equation}\label{eq:grimplane}
\mathrm{G}^n_t:=\{\theta e_1+(t-\log\cos\theta)\varphi\}\,,\quad t\in(-\infty,\infty)\,.
\end{equation}
\end{lemma}
%\begin{lemma}\label{lem:grimconvergence}
%For any sequence of times $t_i\to-\infty$ the sequence $\{F^i\}_{i\in \N}$ of mean curvature flows $F^i:S^n\times(-\infty,-t_i)\to \R^{n+1}$ defined by
%\[
%F^i(x,t):=F(x,t+t_i)-\gamma(0,t_i)
%\]
%converges locally uniformly in the smooth topology to the grim plane $G^\alpha:(-\frac{\pi}{2},\frac{\pi}{2})\times\R^{n-1}\times(-\infty,\infty)\to\R^{n+1}$ defined by
%\begin{equation}\label{eq:grimplane}
%G^\alpha(\theta,z,t)%\in (\pi/2,3\pi/2)\times\R\times \R\mapsto 
%:=\alpha(\theta,\alpha^{-2}t-\log\cos \theta,z)\,,%\in\R^3
%\end{equation}
%where $\alpha$ is defined by \eqref{eq:alpha}. In particular, $\alpha\in(0,1]$. %\begin{equation}\label{translator limit}
%\lim_{t\to-\infty} \Sigma^\infty_t-\gamma_\infty(\pi,t)\to\text{    translator   }
%\end{equation}
%\end{lemma}
\begin{proof}
That a subsequence of the flows converges locally uniformly in $C^\infty$ to a weakly convex eternal limit mean curvature flow contained in the slab $\Omega$ follows similarly as in Theorem \ref{thm:limit}. Now observe that, by the $O(n)$-invariance, the solution $\Sigma_t$ satisfies
\begin{equation*}%\label{eq:lambdabound}
\lambda(0,t)=\frac{1}{\ell(t)}\,,
\end{equation*}
which tends to $0$ as $t$ tends to $-\infty$ because $\Sigma_t$ cannot lie, for all $t<0$, in any compact subset of $\R^{n+1}$. By the strong maximum principle, the limit of the sequence of flows $\{\Sigma^i_t\}_{t\in(-\infty,0)}$ must therefore split off an $(n-1)$-plane (see, for example, \cite[Proposition 4.2.7]{Mantegazza}). Moreover, the limit of the profile curves must satisfy curve shortening flow. Since the curvature of the limit is constant in time with respect to the turning angle parametrization, \cite[Theorem 8.1]{Ham95} shows that the profile curve moves by translation and we conclude that the limit is a Grim hyperplane. Since $\lim_{t\to-\infty} H(\pi,t)=\alpha^{-1}$ exists, as defined in equation~\eqref{eq:alpha}, then the limit flow is the Grim hyperplane $\alpha \mathrm{G}^n_{\alpha^{-2}t}$, independently of the chosen subsequence. Finally, since the limit lies in the slab $\Omega$, $\alpha$ cannot be greater than $1$.
\end{proof}

The remainder of this section is devoted to showing that $\a$ must be equal to 1. We will need the following lemma.

\begin{lemma}\label{lem:k>l}
For all $t<0$
\[
\l(\cdot,t)\le\min\left\{\k(\cdot,t), \frac{\a}{-t}\right\}\,.
\]
\end{lemma}
\begin{proof}
By Lemma \ref{lem:lambdalemma}, $\lambda(\pm\frac{\pi}{2},t)=\kappa(\pm\frac{\pi}{2},t)$ for all $t$ and, by Lemma \ref{lem:grimconvergence},
\[
\lim_{t\to-\infty}\frac{\kappa}{\lambda}(\theta,t)=\infty
\]
for any $\theta\in(-\frac{\pi}{2},\frac{\pi}{2})$. It now follows from the evolution equation \eqref{eq:rotationalevolutionu} and the symmetry of $\gamma$ that $\kappa\geq \lambda$ in $S^1\times(-\infty,0)$.

Next, we claim that
\begin{equation}\label{eq1:l>1/t}
\ell(t)=\la\gamma(\pi,t), e_2\ra\ge-\alpha^{-1}t
\end{equation}
for all $t<0$. Indeed, since the translated profile curves converge to the Grim Reaper of width $\alpha \pi$ and, by the Harnack inequality, $H(\pi,t)$ is non-decreasing in $t$, we find $H(\pi,t)\ge\a^{-1}$ for all $t<0$ and hence
\[
\frac{d}{dt}\la\gamma(\pi,t), e_2\ra=-H(\pi,t)\leq -\alpha^{-1}\,.
\]
Integrating from $t$ to $0$ yields \eqref{eq1:l>1/t}. Furthermore, since we have shown that $\kappa\geq \lambda$, \eqref{eq:Dlambda} shows that  $\l(\theta,t)$ is non-decreasing in $\theta$ for $\theta\in [\tfrac{\pi}{2}, \pi]\cup[\tfrac{3\pi}{2},2\pi]$ and non-increasing in $\theta$ for $\theta\in [0, \tfrac{\pi}{2}]\cup[\pi,\tfrac{3\pi}{2}]$. We conclude that also
\[
\l(\theta,t)\le \l(\pi,t)=\l(0,t)=\frac{1}{\la\gamma(\pi,t), e_2\ra}\le \frac{\a}{-t}\,.
\]
\end{proof}

We now show that the area enclosed by $\Gamma_t$ grows almost like $-2\pi t$ to highest order in $t$.

\begin{lemma}\label{lem:areaupperbound1}
For every $t<0$ and any $\omega\in(0,\tfrac{\pi}{2})$
\[
A(t)\le -(2\pi+4(n-1)\omega)t+2(n-1)(\pi+n\a L(\omega,t)\log(-t))\,,
\]
where $A(t)$ denotes the area enclosed by the profile curve $\gamma(\cdot,t)$  and $L(\omega,t)$ is the length of the segment $\gamma(\cdot,t)|_{\left[\tfrac{\pi}{2}+\omega, \tfrac{3\pi}{2}-\o\right]}$.
\end{lemma}
\begin{proof}
Integrating \eqref{eq:rotationalevolutionarea} and recalling the reflection symmetry of $\Gamma_t$ yields, for any $\omega\in(0,\frac{\pi}{2})$,
\begin{align*}
A(t)={}&-2\pi t+2(n-1)\int_t^0\hspace{-1.5mm}\left(\hspace{-0.5mm}\int_{[0,\frac{\pi}{2}-\omega]\cup[\pi,\frac{3\pi}{2}-\omega]}\frac{\lambda(\theta,\sigma)}{\kappa(\theta,\sigma)}\,d\theta\right)d\sigma\\
{}&+2(n-1)\int_t^0\hspace{-1.5mm}\left(\hspace{-0.5mm}\int_{[\frac{\pi}{2}-\omega,\frac{\pi}{2}]\cup[\frac{3\pi}{2}-\omega,\frac{3\pi}{2}]}\frac{\lambda(\theta,\sigma)}{\kappa(\theta,\sigma)}\,d\theta\right)d\sigma\,.
\end{align*}
Using Lemma \ref{lem:k>l} to estimate 
\begin{equation*}
\lambda\leq
\begin{cases}
\kappa &\text{if }\; -1<t<0\\
\frac{\alpha}{-t}&\text{if }\; t<-1
\end{cases}
\end{equation*}
in the first line and $\lambda\leq\kappa$ in the second yields
\begin{align*}
A(t)\leq{}&-2\pi t+2(n-1)(\pi-2\o t)\\
{}&+2(n-1)\alpha\int_t^{-1}\hspace{-1mm}\frac{1}{-\sigma}\left(\int_{[0,\frac{\pi}{2}-\omega]\cup[\pi,\frac{3\pi}{2}-\omega]}\frac{d\theta}{\kappa(\theta,\sigma)}\right)d\sigma\,.
\end{align*}
The claim follows by estimating, 
%Since $H$ is non-decreasing in time, we can estimate 
for any $t<\sigma<0$,
\[
\k(\theta,t)\le H(\theta,t)\le H(\theta,\sigma)\le n\k(\theta,\sigma)\,.
\]
%and hence
%\begin{equation*}
%A(t) \le -2\pi t+4(n-1)\o(1-t)+2(n-1)\a\log(-t)L(\omega,t)\,,
%\end{equation*}
%where $L(\omega,t)$ is the length of the segment $\gamma(\cdot,t)|_{\left[\tfrac{\pi}{2}+\omega, \tfrac{3\pi}{2}-\o\right]}$. 
\end{proof}

We now have all the ingredients needed to prove that $\alpha=1$.
\begin{lemma}\label{lem:a=1}
$\a=1$.
\end{lemma}
\begin{proof}
First note that, by the convergence of the edge region to the Grim hyperplane (Lemma \ref{lem:grimconvergence}), the length of any fixed segment $\gamma([\frac{\pi}{2}+\omega,\frac{3\pi}{2}-\omega],t)$ is bounded as $t\to-\infty$. Thus, the area estimate of Lemma \ref{lem:areaupperbound1} provides, for any $\omega\in(0,\tfrac{\pi}{2})$, some time $t_\omega<0$ such that
\[
A(t)\leq -(2\pi+5(n-1)\omega)t\quad\text{for all}\quad t<t_\omega\,.
\]
Recalling \eqref{eq1:l>1/t}, this becomes
\begin{equation}\label{eq:areaabove}
A(t)\leq (2\pi+5(n-1)\omega)\alpha\ell(t)\quad\text{for all}\quad t<t_\omega\,.
\end{equation}
On the other hand, for any $\omega\in(0,\tfrac{\pi}{2})$, we can estimate the enclosed from area below by (see figure \ref{fig:areas})
\[
\begin{split}
\frac{1}{2}A(t)\geq \ell(t)\cdot (x(\tfrac{\pi}{2}-\omega,t)-x(\tfrac{3\pi}{2}+\o,t))&+A^+(\omega,t)-A^-(\omega,t)\,, 
\end{split}
\]
where, setting $h_+(t):=x(\tfrac\pi2,t)$ and $h_-(t):=x(-\tfrac{\pi}2, t)$,{\small
\[
\begin{split}
A^+(\o, t):={}&\frac{1}{2}\left(\ell(t)-\left[y\left(\tfrac{\pi}{2}-\omega,t\right)-y(0,t)\right]\right)\left(h_+(t)-x\left(\tfrac{\pi}{2}-\omega,t\right)\right)\\
&+\frac{1}{2}\left(\ell(t)-\left[y\left(-\tfrac{\pi}{2}+\omega,t\right)-y(0,t)\right]\right)\left(x\left(-\tfrac{\pi}{2}+\omega,t\right)-h_-(t)\right),
\end{split}
\]
}and
\[
A^-:=\int_{-\frac{\pi}{2}+\omega}^{\frac{\pi}{2}-\omega}\left[y(\theta,t)-y(0,t)\right]\frac{\cos\theta}{\k(\theta,t)}d\theta\,.%+\int^{2\pi}_{\tfrac{3\pi}{2}+\omega}\left[y(\theta,t)-y(0,t)\right]\frac{\cos\theta}{\k(\theta,t)}d\theta\,.
\]
\begin{figure}
\begin{center}
\includegraphics[width=0.6\textwidth]{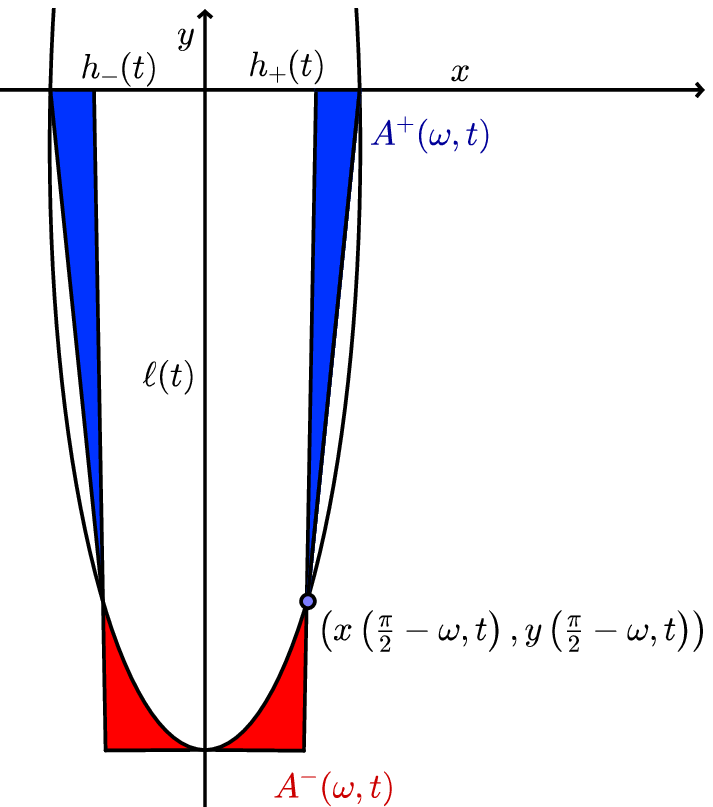}
\end{center}
\caption{Bounding the enclosed area from below.}\label{fig:areas}
\end{figure}

\noindent But, as the segment $\gamma(\cdot, t)|_{\left[-\frac{\pi}{2}+\omega,\frac{\pi}{2}-\omega\right]}$ converges to the corresponding segment of the Grim Reaper of width $\alpha$ as $t\to-\infty$, there exists some time $t_{\omega,\alpha}\le t_\o$ such that for all  $t<t_{\omega,\alpha}$
\[
\alpha\pi\geq x\left(\tfrac{\pi}{2}-\omega,t\right)-x\left(-\tfrac{\pi}{2}+\omega,t\right)\geq \alpha\left({\pi}-3\omega\right)\,,
\]
\[
\max\{y\left(\tfrac{\pi}{2}-\omega,t\right)-y(0,t)\,,\,y\left(-\tfrac{\pi}{2}+\omega,t\right)-y(0,t)\}\leq c_1(\omega,\alpha)
\]
\[
A^-(\omega,t)\leq c_2(\omega,\alpha)
\]
and, since $\gamma$ is contained in the slab $\{(x,y)\in\R^2:-\tfrac\pi2<x<\tfrac\pi2\}$ and in no smaller slab,
\[
\pi-\alpha\omega \leq h_+(t)-h_-(t)\leq \pi\,.
\]
Putting these together yields
%\[
%\frac{1}{4}A(t)\geq \alpha\left(\frac{\pi}{2}-2\omega\right)\cdot \ell(t)+
%\frac{1}{2}(\ell(t)-c_1)\cdot(h(t)-\alpha\pi/2)-c_2\,. 
%\]
%That is,
\begin{equation}\label{eq:arealowerbound}
%(2\pi+5\omega)(1+\alpha\,\ell(t))\geq 
A(t)%\geq \alpha\left(2\pi-8\omega\right)\cdot \ell(t)+2(\ell(t)-c_1)\cdot\left(h(t)-\frac{\alpha\pi}{2}\right)-4c_2\,. 
\geq (2\pi\alpha-7\alpha\omega+(1-\alpha)\pi)\ell(t)-\pi c_1-4c_2\,.
\end{equation}
Recalling \eqref{eq:areaabove}, we obtain
\[
0\geq\left[(1-\alpha)\pi-(7+5(n-1))\alpha\omega\right]\ell(t)-\pi c_1-4c_2\,.
%0\geq\left(2h(t)-\alpha\pi-13\alpha\omega\right)\ell(t)+\pi c_1-4c_2-2\pi-5\omega\,. 
\]
Choosing $\omega$ so small that $(1-\alpha)\pi>(7+5(n-1))\alpha\omega$ and letting $t$ tend to $-\infty$ yields a contradiction, unless $\alpha=1$.
\end{proof}

This yields the desired asymptotics for the edge region.
  
 \begin{corollary}\label{cor:asymptoticsedge}
For any sequence of times $t_i\to-\infty$ and any unit vector $\varphi\in \{e_1\}^\perp$ the sequence of mean curvature flows $\{\Sigma^i_t\}_{t\in (-\infty,-t_i)}$ defined by $\Sigma^i_t:=\Sigma^i_{t+t_i}-P(\varphi,t_i)$ converges locally uniformly in the smooth topology to the Grim hyperplane $\mathrm{G}^n_t$ defined by \eqref{eq:grimplane}.
\end{corollary}

%\begin{lemma}\label{lem:asymptoticsedge}
%For any sequence of times $t_i\to-\infty$ the sequence $\{F^i\}_{i\in \N}$ of mean curvature flows %$F^i:S^n\times(-\infty,-t_i)\to \R^{n+1}$ defined by
%\[
%F^i(x,t):=F(x,t+t_i)-\gamma(0,t_i)
%\]
%converges locally uniformly in the smooth topology to the Grim hyperplane $G:(-\frac{\pi}{2},\frac{\pi}{2})\times\R^{n-1}\times(-\infty,\infty)\to\R^{n+1}$ defined by
%\begin{equation*}%\label{eq:grimplane}
%G(\theta,z,t):=(\theta,t-\log\cos \theta,z)\,.
%\end{equation*}
%\end{lemma}

\section{Reflection symmetry}\label{sec:reflection}

Combining our knowledge of the asymptotics of the solution with the Alexandrov reflection principle, we will show that the solution must be reflection symmetric with respect to the hyperplane $\{x_1=0\}$. We remark that the Alexandrov reflection principle has been used by Chow and Gulliver \cite{Chow97,ChGu01} to prove gradient estimates for geometric flows and, recently, by Bryan and Louie to prove uniqueness of convex ancient curve shortening flows on the sphere \cite{BrLo}. We will apply it in a novel way in Section \ref{sec:uniqueness} to prove the uniqueness result of Theorem \ref{thm:uniqueness}.

Let us begin by recalling the reflection principle: Let $\Sigma^n$ be any smooth, embedded hypersurface of $\R^{n+1}$ which bounds an open set $\Omega\subset \R^{n+1}$. Given a unit vector $e\in S^n$ and some $\alpha\in \R$, denote by $\Pi_{e,\alpha}$ the halfspace $\{p\in \R^{n+1}:\left\langle p,e\right\rangle<\alpha\}$ and by $R_{e,\alpha}\cdot\Sigma:=\{p-2(\left\langle p,e\right\rangle-\alpha)e:p\in \Sigma\}$ the reflection of $\Sigma$ across the hyperplane $\partial \Pi_{e,\alpha}$. Following Chow \cite{Chow97}, we say that $\Sigma$ \emph{can be reflected strictly about $\Pi_{e,\alpha}$} if $(R_{e,\alpha}\cdot\Sigma)\cap \Pi_{e,\alpha}\subset \Omega\cap\Pi_{e,\alpha}$. The Alexandrov reflection principle can be stated as follows.

\begin{lemma}[Alexandrov reflection principle]\label{lem:Alexandrov}
Let $\{\Sigma_t\}_{t\in [t_0,0)}$ be a smooth, embedded solution of mean curvature flow. If $\Sigma_{t_0}$ can be reflected strictly about $\Pi_{e,\alpha}$
then $\Sigma_t$ can be reflected strictly about $\Pi_{e,\alpha}$ for all $t\in [t_0,0)$.
\end{lemma}
\begin{proof}
This is a consequence of the strong maximum principle and the boundary point lemma for strictly parabolic equations. See \cite[Theorem 2.2]{Chow97}.
\end{proof}

\begin{theorem}\label{thm:reflection symmetry}
Let $\{\Sigma_t\}_{t\in (-\infty,0)}$ be a compact, convex, $O(n)$-invariant ancient solution of mean curvature flow in $\R^{n+1}$ that lies in the slab $\Omega:=\{(x,y,z)\in \R\times\R\times\R^{n-1}:\vert x\vert<\tfrac{\pi}{2}\}$ (and in no smaller slab). Then $\Sigma_t$ is reflection symmetric in the plane $\{(x,y,z)\in \R\times\R\times\R^{n-1}:x=0\}$ for all $t<0$.
\end{theorem}

\begin{proof}
Consider, for any $\alpha\in(0,\tfrac\pi2)$, the halfspace $\Pi_\alpha:=\{(x,y,z)\in \R\times\R\times\R^{n-1}:x<\alpha\}=\Pi_{e_1,\alpha}$. %Set $\Sigma_t^\alpha:=\Sigma_t\cap \Pi_\alpha$ and let $R_\alpha\cdot\Sigma_t^\alpha:=\{(2\alpha-x,y,z):(x,y,z)\in \Sigma^\a_t\}$ be the reflection of $\Sigma_t^\a$ across $\Pi_\a$.

\begin{claim}\label{claim:reflectionII}
For every $\alpha\in(0,\frac{\pi}{4})$ there exists $t_\alpha>-\infty$ such that $\Sigma_t$ can be reflected strictly about $\Pi_{\alpha}$ for all $t<t_\alpha$. %$\Sigma_t\cap R_\alpha\cdot\Sigma^\alpha_t=\emptyset$ for all $t<t_\alpha$.
\end{claim}
\begin{proof}[Proof of Claim \ref{claim:reflectionII}]
Suppose that the claim does not hold. Then there must be some $\alpha\in(0,\frac{\pi}{4})$ and a sequence of times $t_i\to-\infty$ such that $(R_{\alpha}\cdot\Sigma_{t_i})\cap \Pi_\alpha\cap \Sigma_{t_i}\neq\emptyset$. Choose a sequence of points $p_i=x_ie_1+y_ie_2\in\Gamma_{t_i}$ whose reflection about the hyperplane $\Pi_\alpha$ satisfies $(2\alpha-x_i)e_1+y_ie_2\in (R_{\alpha}\cdot\Gamma_{t_i})\cap\Gamma_{t_i}\cap \Pi_\alpha$, where $\Gamma_t$ denotes the profile curve of $\Sigma_t$ in $\textrm{span}\{e_1,e_2\}$, and set $p'_i=x'_ie_1+ y'_ie_2:=(2\alpha-x_i)e_1+y_ie_2$. Without loss of generality, we may assume that $y'_i=y_i\geq 0$. Since $\alpha\leq x_i<\frac{\pi}{2}$, the point $p'_i$ satisfies  $\alpha\geq x'_i>-\frac{\pi}{2}+2\alpha$ so that, after passing to a subsequence, $\theta':=\lim_{i\to\infty}x'_i\in[-\frac{\pi}{2}+2\alpha,\alpha]$. Then, by Corollary \ref{cor:asymptoticsedge},
\[
\lim_{i\to \infty}(\langle P(e_2, t_i), e_2\rangle-y_i)=\lim_{i\to \infty}(\langle P(e_2, t_i), e_2\rangle-y'_i)= -\log\cos \theta'\,,
\]
where recall that $P$ is the inverse of the Gauss map. But then, again by Corollary~\ref{cor:asymptoticsedge},
\[
\lim_{i\to\infty}x_i=-\lim_{i\to\infty}x'_i=-2\alpha+\lim_{i\to\infty}x_i\,.
\]
But this implies that $\alpha=0$, a contradiction.
\end{proof}

It now follows from Lemma \ref{lem:Alexandrov} that $\Sigma_t$ can be reflected across $\Pi_\alpha$ for all $t<0$. The same argument applies when the halfspace $\Pi_\alpha$ is replaced by $-\Pi_{\alpha}=\{(x,y,z)\in \R\times\R\times\R^{n-1}:x>-\a\}$. Now take $\alpha\to 0$.

\end{proof}

\section{Area and displacement estimates}\label{sec:area estimates}

We now continue our study of convex $O(n)$-invariant ancient solutions lying in a slab. Our aim in this section is to derive refined estimates for the area $A(t)$ of the regions enclosed by the curves $\Gamma_t$ and for their vertical displacement $\ell(t)$. The latter estimate is the final ingredient needed to prove the uniqueness of the solution constructed in Theorem \ref{thm:limit}.

Since, by Theorem \ref{thm:reflection symmetry}, we know that any such solution is reflection symmetric with respect to the hyperplane $\{x_1=0\}$, the horizontal displacement
\[
h(t):=\max_{\theta\in S^1}\la\gamma(\theta,t), e_1\ra
\]
satisfies
\[
h(t)=\la\gamma(\tfrac{\pi}{2},t), e_1\ra=-\la\gamma(-\tfrac{\pi}{2},t), e_1\ra\,.
\]

We also note that the minimum value of the mean curvature occurs at the poles,
\begin{equation}\label{eq:Hminatpole}
\min_{\theta\in S^1} H(\theta,t)=H(\pm\tfrac{\pi}{2}, t)\quad\text{for all}\quad t<0\,.
\end{equation}
Indeed, by Lemma \ref{lem:k>l} and the identity \eqref{eq:Dlambda}, we see that $\l(\cdot, t)$ is non-decreasing in $(\tfrac{\pi}{2},\pi)$. Thus, using once more Lemma \ref{lem:k>l} and the $O(n)$-invariance,
\[
\k(\theta,t)\ge \l(\theta, t)\ge\l(\tfrac{\pi}{2},t)=\k (\tfrac{\pi}{2},t)
\]
for all $\theta\in(\frac{\pi}{2},\pi)$. By the reflection symmetries of $\gamma$, this proves \eqref{eq:Hminatpole}.

We want to obtain a better bound for $A(t)$ than the one provided in Lemma \ref{lem:areaupperbound1}. To do so, we need a better estimate for the integral of $\l(\cdot, t)$. Note that the part of the curve $\Gamma_t$ with $y>0$ can be written as a graph $(x,u(x,t))$ with $x\in [-h(t),h(t)]$. With respect to this graphical representation, the integral of $\lambda$ becomes
\begin{equation}\label{eq:lu}
\frac{1}{4}\int_{\Gamma_t}\l\,ds=\int_{0}^{h(t)}\frac{1}{u(x,t)}dx\,,
\end{equation}
where $s=s(t)$ is the arc-length parameter for $\Gamma_t$. So we would like to estimate the graph height $u$ from below. Note that $u$ evolves according to
\[
\frac{du}{dt}=\frac{u_{xx}}{1+(u_x)^2}-\frac{n-1}{u} =-H\sqrt{1+u_x^2}\,.
\]
Since $H$ is non-decreasing in $t$ and, by the edge asymptotics (Corollary \ref{cor:asymptoticsedge}), tends to $\vert{\cos \theta}\vert$ as $t$ tends to $ -\infty$, we can estimate
\begin{equation*}%\label{evolutionofu}
-\frac{du}{dt}\ge 1\,.
\end{equation*}
Given $x\in [0,\frac{\pi}{2})$, denote by $T(x)\in(-\infty,0)$ the time at which the profile curve passes through the point $xe_1$. Then $h(T(x))=x$ and hence $u(x,T(x))=0$ so that, integrating the preceding inequality between $T$ and $t$,
%\begin{lemma} \label{lem:ulowerbound}
%For any $t<0$ and any $x\in[0, h(t))$
\begin{equation*}%\label{eq:ulowerbound}
u\left(x, t\right) \ge -t+T(x)\quad\text{for all}\quad x\in [0,h(t))\,.%\ge-t+\frac{4}{\log(\tfrac{2x}{\pi})}\ge \frac{2\pi}{\frac{\pi}{2}-x}\,.
\end{equation*}
%\end{lemma}
%\begin{proof}
%Fix any $t<0$ and any $x\in[0,h(t))$ and set $t_1:=\arc{h}(x)>t$, where $\arc{h}:[0,\pi/2)\to(-\infty,0]$ is the inverse function of $h$. Then integrating the estimate \eqref{evolutionofu} between $t$ and $t_1$ yields
%\[
%\begin{split}
%u(x, t)=-\int_{t_0}^{t_1}\frac{du(h(t_1),t)}{dt} dt\ge t_1-t=-t+ \arc{h}(x)\,.
%\end{split}
%\]
%Now estimate
%\begin{equation}\label{hinv}
%\arc{h}(x)\ge \frac{4}{\log(\tfrac{2x}{\pi})}
%\end{equation}
%using Lemma \ref{domain bound}.
%\end{proof}

We will estimate $T(x)$ from below by estimating $h(t)$ similarly as in Lemmas \ref{Hmin-lemma} and \ref{h-lemma}. Fix $x\in [0,\frac{\pi}{2})$ and consider the circle $\mathcal{C}(x,t)$ which passes through the points $\gamma(\frac{\pi}{2},t)=(h(t),0)$ and $(x,u(x,t))$. Then, proceeding as in the proof of Lemma \ref{Hmin-lemma}, we find that $\Gamma_t$ lies locally outside of $\mathcal{C}(x,t)$ near $(h(t),0)$ and hence
\begin{equation}\label{eq:Hminbound}
H_{\min}(t)\le \frac{2n(h(t)-x)}{u(x,t)^2+(h(t)-x)^2}\le \frac{2n(h(t)-x)}{u(x,t)^2}\,.
\end{equation}

On the other hand, estimating $H(0,t)\geq\vert{\cos 0}\vert=1$,
\begin{equation*}%\label{lt}
\ell(t)=-\int_t^0H(0,s)\,ds\ge -t
\end{equation*}
so that, setting $x=0$ in \eqref{eq:Hminbound},
\begin{equation}\label{eq:Hminboundk=1}
H_{\min}(t)\leq \frac{n\pi}{(-t)^2}\,.
\end{equation}
This then implies
\[
-\frac{dh}{dt}=H_{\min}\leq \frac{2nh}{h^2+\ell^2}\leq \frac{2nh}{(-t)^2}\,,
\]
%\[
%\frac{d(-\log(h(t)))}{dt}\le \frac{d}{dt}\left(\frac{2n}{-t}\right)
%\]
which, after integrating, yields
\[
h(t)\ge h(T_0)\exp\left(\frac{2n}{t}-\frac{2n}{T_0}\right)\quad\text{for any}\quad T_0<t<0
\]
and hence, after taking $T_0\to-\infty$,%an analogue of Lemma \ref{h-lemma},
\begin{equation}\label{eq:hboundk=1}
h(t)\ge\tfrac\pi2 \mathrm{e}^{\frac{2n}{t}}\ge \frac{\pi}{2}\left(1-\frac{2n}{-t}\right)\quad\text{for all}\quad t<0\,.
\end{equation}
Setting $t=T(x)$, this now provides an estimate for $T(x)$ and hence $u(x,t)$:
\begin{equation}\label{eq:uboundk=1}
u\left(x, t\right) \ge -t+T(x)\ge-t-\frac{n\pi}{\frac{\pi}{2}-x}\,.
\end{equation}
We can now pass this estimate back into \eqref{eq:Hminbound}! By carefully choosing the point $x$, the process can be iterated.% to prove the following lemma. %Iterating in this way allows us to greatly improve the estimates \eqref{eq:Hminboundk=1}, \eqref{eq:domainbound} and \eqref{eq:ulowerbound}.

\begin{lemma}\label{lem:uHhk} 
For every $k\in \N$ there exists a constant $c_k$ such that, for all $t<0$ and all $x\in[0, h(t)]$,
\begin{enumerate}[(i)]
\item $\displaystyle H_{\min}(t)\le\frac{n\pi kc_k}{(-t)^{k+1}}$,
\item $\displaystyle h(t)\ge\frac\pi2\left(1-\frac{2nc_k}{(-t)^{k}}\right)$ and
\item $\displaystyle u(x, t) \ge-t-\left(\frac{n\pi c_k}{\frac{\pi}{2}-x}\right)^\frac{1}{k}$.
\end{enumerate}
\end{lemma}
\begin{proof}
Since $H(\cdot,t)\sim(-t)^{-\frac{1}{2}}$ \cite{Hu84} and $u(\cdot,t)\to 0$ as $t\to 0$ \cite{Hu84}, it suffices to find, for each $k\in\N$, some $t_k\leq 0$ such that the claims hold for all $t<t_k$. Note also that the estimates \eqref{eq:Hminboundk=1}, \eqref{eq:hboundk=1} and \eqref{eq:uboundk=1} prove the claim in the case $k=1$. So suppose that the claim holds for some integer $k\geq 1$. Given $t<0$ choose $\underline x=\underline x(t)$ so that
\[
\frac{\pi}{2}-\underline x=\frac{2n\pi c_k}{(-t)^k}\,.
\]
Then, when $t\leq t_k:=-(4nc_k)^{\frac{1}{k}}$, $\underline x\geq 0$ and, by the induction hypothesis on $h$, $\underline x\le h(t)$ and hence, by the induction hypothesis on $u$, $u(\underline x,t)\geq -(1-2^{-\frac{1}{k}})t$. Recalling \eqref{eq:Hminbound}, we then obtain
\[
H_{\min}(t)\leq\frac{n\pi(k+1)c_{k+1}}{(-t)^{k+2}}\,,
\]
where $c_{k+1}:=\frac{2nc_k}{(1-2^{-\frac{1}{k}})(k+1)}$. Integrating in time from $-\infty$ to $t$ yields
\[
h(t)\geq \frac{\pi}{2}\left(1-\frac{2n c_{k+1}}{(-t)^{k+1}}\right)\,.
\]
Finally, setting $t=T(x)$ and recalling \eqref{eq:uboundk=1},
\[
u(x, t) \ge-t-\left(\frac{n\pi c_k}{\frac{\pi}{2}-x}\right)^\frac{1}{k}\,.
\]
The claim now follows by induction.
\end{proof}

We can now prove the desired area estimate.

\begin{corollary}\label{cor:areabounds} There exists a constant $C$ such that
\[
-t+(n-1)\log(-t)+C\ge \frac{A(t)}{2\pi}\ge-t+(n-1)\log(-t)-C
\]
for all $t\in(-\infty,1)$.
\end{corollary}
\begin{proof}
The claims will follow by estimating the integral of $\lambda$ appropriately.
\begin{claim}\label{claim:lambdaL1est}
For any $\varepsilon>0$
\[
\int_{\Gamma_t}\l\,ds\leq\frac{2\pi}{-t}+o\left(\frac{1}{(-t)^{2-\varepsilon}}\right).
\]
\end{claim}
\begin{proof}[Proof of Claim \ref{claim:lambdaL1est}]
Let $c(x,t)=\sqrt{\rho^2(t)-(x-a(t))^2}$ be the graph of the circle of radius $\rho(t):=\frac{1}{\pi}(-t)^2$ and centre $(a(t),0)$, where $a(t):=-\rho(t)+h(t)$. Note that $c$ and $u$ are tangent at the point $(h(t),0)$. Set 
\[
\underline x(t):=\inf\{x\in [\tfrac{\pi}{4},h(t)]:u(x,t)=c(x,t)\}\,.
\]
Then $u$ lies above $c$ in $[0,\underline x(t))$ and hence, recalling \eqref{eq:lu},
\begin{align*}
\frac{1}{4}\int_{\Gamma_t}\lambda\,ds={}&\int_0^{\underline x(t)}\frac{dx}{u(x,t)}+\int_{\arc \underline x(t)}\lambda\,ds\,,%\\
%\leq{}& \int_0^{\underline x(t)}\frac{dx}{c(x,t)}+\int_0^{\underline s(t)}\lambda(s,t)\,ds\,,
\end{align*}
where, for any $x\in [0,h(t)]$, $\arc x$ is the arc of $\Gamma_t$ joining the point $(h(t),0)$ to the point $(x,u(x))$. Set
\[
\xi(t):=h(t)-\underline x(t)\,.
\]
By the monotonicity of $\lambda$ and the convexity of $\Gamma_t$, we may estimate
\begin{align}
\int_{\arc \underline x(t)}\lambda\,ds\leq \lambda(\underline x(t),t)\,\underline s(t)\leq{}&-\frac{\cos\underline \theta^u(t)}{u(\underline x(t),t)}\left(\xi(t)+u(\underline x(t),t)\right)\nonumber\\
\leq{}&-\cos\underline \theta^c(t)\left(1+\frac{\xi(t)}{c(\underline x(t),t)}\right)\,,\label{eq:arcxestimate}
\end{align}
where $\underline s(t)$ is the length of $\arc \underline x(t)$ and $\underline \theta^u$ resp. $\underline\theta^c$ is the angle between the tangent vector to $u$ resp. $c$ and the $x$-axis at $\underline x(t)$. We claim that
\begin{equation}\label{eq:xiest}
\xi(t)\leq \frac{C_k}{(-t)^k}
\end{equation}
for any $k\in\N$, where $C_k:=2n c_k\left(1-\frac{1}{\sqrt{2}}\right)^{-k}$. %First note that 
%\[
%c^2(0,t)=2\rho h(t)-h(t)^2=\left(\frac{2}{\pi}(-t)^2-h(t)\right)h(t)\leq \frac{(-t)^2}{2}\leq \frac{\ell^2(t)}{2}
%\]
%and hence, by convexity of $\Gamma_t$,
%\[
%\xi(t)\leq \frac{h(t)}{2}\leq \frac{\pi}{4}\,.
%\]
%\[
%\xi(t)\geq \frac{\pi}{2}\varepsilon^2
%\]
Indeed, by Lemma \ref{lem:uHhk},
\begin{align*}
-t-\left(\frac{n\pi c_k}{\xi(t)}\right)^{\frac{1}{k}}\leq u(\underline x(t),t)={}&c(\underline x(t),t)\leq{}\sqrt{2\rho\xi(t)}=-t\sqrt{\frac{2}{\pi}\xi(t)}\,.
\end{align*}
Rearranging yields
\begin{align*}
\frac{2}{\pi}\xi(t)\left(1-\sqrt{\frac{2}{\pi}\xi(t)}\right)^k\leq\frac{2nc_k}{(-t)^k}\,.
\end{align*}
%and hence
%\begin{align*}
%\xi(t)\leq%\frac{2\pi c_k}{\left(1-\frac{1}{\sqrt{2}}\right)^k}
%\frac{C_k}{(-t)^k}
%\end{align*}
%or
%\begin{align*}
%\xi(t)\geq \frac{\pi}{2}\left(1-\frac{C}{-t}\right)^2\,.
%\end{align*}
The claim follows since, by assumption, $\xi(t)\leq \frac{\pi}{4}$. We can now use \eqref{eq:xiest} to estimate
\[
\frac{\xi(t)}{c(\underline x(t),t)}=\sqrt{\frac{\xi(\underline x(t),t)}{\frac{(-t)^2}{\pi}-\xi}} \le o\left(\frac{1}{(-t)^k}\right)
\]
and
\[
-\cos\underline\theta^c=\frac{1}{\sqrt{1+ c'(\underline x(t))^2}}=\frac{1}{\sqrt{1+\frac{(\r-\xi(t))^2}{2\r\xi(t)- \xi(t)^2}}}\le o\left(\frac{1}{(-t)^k}\right)
\]
for any $k\in \N$. Recalling \eqref{eq:arcxestimate}, we conclude that
\begin{align}\label{eq:lambdaestimatenearpole}
\int_{\arc \underline x(t)}\lambda\,ds\leq %\lambda(\underline x(t),t)\,\underline s(t)\leq
 o\left(\frac{1}{(-t)^k}\right)
\end{align}
for any $k\in \N$.

Since $u$ lies above $c$ in $[\frac{\pi}{4},\underline x(t)]$, we may estimate the remaining term by
\begin{align*}
\int_0^{\underline x(t)}\frac{dx}{u(x,t)}\leq{}& \int_0^{x_0(t)}\frac{dx}{u(x,t)}+\int_{x_0(t)}^{h(t)}\frac{dx}{c(x,t)}\\
\leq{}&\frac{x_0(t)}{u(x_0(t),t)}+\int_{x_0(t)}^{h(t)}\frac{dx}{\sqrt{\r^2-(x-a)^2}}\,,
\end{align*}
where $x_0(t)\geq \frac{\pi}{4}$ will be chosen in a moment. Note that
\begin{align*}
%={}&\int_{1-\tfrac {h(t)}{\r}}^{1}\frac{dx}{\sqrt{1-x^2}}\\
\int_{x_0(t)}^{h(t)}\frac{dx}{\sqrt{\r^2-(x-a)^2}}={}&\frac\pi2-\arcsin\left(1-\frac{h(t)-x_0(t)}{\r}\right)\\
\leq{}&\sqrt{\frac{2(h(t)-x_0(t))}{\rho}}+o\left(\frac{h(t)-x_0(t)}{\rho}\right)\\
={}&\frac{\sqrt{2\pi(h(t)-x_0(t))}}{-t}+o\left(\frac{1}{(-t)^2}\right)\,.
%\leq{}&\frac{\pi}{-2t}+o\left(\frac{1}{-t}\right)
\end{align*}
Now choose $x_0(t)$ so that $\frac{\pi}{2}-x_0(t)=\frac{\pi nc_k}{(-t)^k}$. Then, by Lemma \ref{lem:uHhk}, $0\leq x_0(t)\leq h(t)$ for $t\leq t_k<0$, say, and hence
\[
u(x_0(t),t)\geq -t-c_{k,l}(-t)^{\frac{k}{l}}
\]
for any $l>k\in \N$, where $c_{k,l}:=(c_l/c_k)^{\frac{k}{l}}$. We conclude that
\[
h(t)-x_0(t)\leq \frac{\pi}{2}-x_0(t)=\frac{\pi nc_k}{(-t)^{k}}
\]
and
\begin{align*}
\frac{x_0(t)}{u(x_0(t),t)}\leq \frac{\pi}{-2t}\frac{1}{1-c_{k,l}(-t)^{\frac{k}{l}-1}}\leq{}&\frac{\pi}{-2t}\left(1+c_{k,l}(-t)^{\frac{k}{l}-1}\right)\\
%\leq{}&\frac{\pi}{-2t}+o\left(\frac{1}{-t}\right)
\end{align*}
for any $l>k\in \N$ and hence
\begin{align*}
\int_0^{\underline x(t)}\frac{dx}{u(x,t)}\leq{}&\frac{\pi}{-2t}+o\left(\frac{1}{(-t)^{2-\varepsilon}}\right)
\end{align*}
for any $\varepsilon>0$. Combined with \eqref{eq:lambdaestimatenearpole}, this completes the proof of the claim.
\end{proof}

Integrating \eqref{eq:rotationalevolutionarea} and applying Claim \ref{claim:lambdaL1est} now yields, on the one hand,
%\begin{equation}\label{PDEA1}
%-\frac{d}{dt}A(t)\le 2\pi\left(1+\frac{1}{-t}\right)+o\left(\frac1t\right)\,.
%\end{equation}
%Because of Remark \ref{rmk:lupperbound}, we can integrate \eqref{PDEA1}, which yields
\begin{equation}\label{eq:areaup1}
\begin{split}
A(t)\le -2\pi t+2\pi(n-1)\log(-t)+2\pi C\quad\text{for all}\quad t<-1
\end{split}
\end{equation}
for some constant $C$.

On the other hand, applying H\"older's inequality, we can estimate
\[
h(t)^2\le \int_{0}^{h(t)}{u(x)}dx\int_{0}^{h(t)}\frac{1}{u(x)}dx%=\frac{A(t)}{4}\int_{0}^{h(t)}\frac{1}{u(x)}dx
=\frac{A(t)}{16}\int_0^{L(t)}\lambda(s,t)\,ds
\]
so that, bounding $h$ from below using Lemma \ref{lem:uHhk},
\[
\begin{split}
\int_{\Gamma_t}\l(s,t)ds&\ge \frac{2\pi \left(1-\frac{2nc_1}{-t}\right)^2}{-t+(n-1)\log(-t)+ {C}}\ge \frac{2\pi}{-t}-o\left(\frac{1}{(-t)^{2-\varepsilon}}\right)
\end{split}
\]
for any $\varepsilon>0$. Integrating \eqref{eq:rotationalevolutionarea}, we conclude that
\begin{equation}\label{eq:areadown1}
\begin{split}
A(t)\ge& -2\pi t+2\pi(n-1)\log(-t)- 2\pi C\quad\text{for all}\quad t<-1
\end{split}
\end{equation}
for some $C\geq 0$.
\end{proof}

The area estimates of Corollary \ref{cor:areabounds} allow us to prove the following displacement bound.

\begin{lemma}\label{lem:diameterbounds}
For any $\varepsilon\in(0,1)$,
\[
\ell(t)\leq -t+o((-t)^\varepsilon)\,.
\]
\end{lemma}
\begin{proof}
Given $t<0$ and $k\in\N$ set $\underline x(t):=\frac{\pi}{2}-\frac{n\pi c_k}{(-t)^k}$ so that, by Lemma \ref{lem:uHhk}, $h\geq \underline x\geq 0$ (for all $t<t_k$, say). Then we may estimate the area under the graph of $u(\cdot,t)$ from below by
\begin{align*}
\frac{1}{4}A(t)\geq{}& u(\underline x(t),t)\underline x(t)+\frac{1}{2}\left(\ell(t)-u(\underline x(t),t)\right)\underline x(t)\\
={}&\frac{1}{2}\left(\ell(t)+u(\underline x(t),t)\right)\underline x(t)\\
={}&\frac{\pi}{4}\left(\ell(t)+u(\underline x(t),t)\right)\left(1-\frac{2nc_k}{(-t)^k}\right)\,.
\end{align*}
Recalling Lemma \ref{lem:uHhk} once more, we may estimate $u(\underline x(t),t)$ from below by
\[
u(\underline x(t),t)\geq -t-\left(\frac{n\pi c_l}{\frac{\pi}{2}-\underline x(t)}\right)^{\frac{1}{l}}=-t-c_{k,l}(-t)^{\frac{k}{l}}
\]
for any $l\in\N$, where $c_{k,l}:=(\frac{c_l}{c_k})^{\frac{1}{l}}$ and hence, bounding $A(t)$ from above as in Corollary \ref{cor:areabounds},{\small
\begin{align*}
\ell(t)\left(1-\frac{2nc_k}{(-t)^k}\right)\leq{}& -t\left(1-\frac{2nc_k}{(-t)^k}\right)+4nc_k(-t)^{1-k}\\
{}&+2((n-1)\log(-t)+C)+c_{k,l}(-t)^{\frac{k}{l}}\left(1-\frac{2nc_k}{(-t)^k}\right)\,.
\end{align*}}

\noindent The claim follows since we may choose $l=rk$ with $r\in\N$ arbitrarily large.
\end{proof}

%We will now use the diameter estimate of Lemma \ref{lem:diameterbounds} to prove some refined estimates for the mean curvature. These estimates will allow us to improve on the asymptotics provided in Lemma \ref{cor:asymptoticsedge}. In the following lemma we show that we can improve the estimate on the mean curvature $H(\theta, t)\ge |\cos\theta|$ which is a result of the Harnack inequality.
As a consequence, we are able to obtain a slightly better lower bound for $H$ than the crude estimate $H(\theta,t)\geq \vert{\cos\theta}\vert$ provided by the Harnack inequality.

\begin{corollary}\label{cor:improvedHestimate} For any $\varepsilon\in(0,1)$
\[
H(\theta, t)\ge \vert{\cos\theta}\vert\left(1+\frac{n-1}{-t}-o\left(\frac{1}{(-t)^{2-\varepsilon}}\right)\right)\;\;\text{for all}\;\; \theta \in S^1\,.
\]
\end{corollary}

\begin{proof} By symmetry, it suffices to prove the claim for $\theta\in (-\pi/2,\pi/2)$ (note that for $\theta=\pm\pi/2$ the claim holds trivially). Consider the function $w:(-\frac{\pi}{2},\frac{\pi}{2})\times(-t,0)\to\R$ defined by $w(\theta,t):=f(t)\cos\theta$, where the function $f:(-\infty,0)\to\R_+$ will be determined momentarily. Observe that
\[
w_t=\kappa^2w_{\theta\theta}+|\mathrm{II}|^2w+\left(\frac{f'}{f}-(n-1)\lambda^2\right)w\,,
\]
where $\vert \mathrm{II}\vert^2:=\kappa^2+(n-1)\lambda^2$.
Recalling \eqref{eq:rotationalevolutionH}, we compute
\begin{align*}
%\frac{w}{H}\left(\left(\frac{H}{w}\right)_t-\kappa^2\left(\frac{H}{w}\right)_{\theta\theta}\right)
\frac{w}{H}\left(\partial_t-\kappa^2\partial^2 _{\theta}\right)\frac{H}{w}={}&2\frac{w}{H}\left(\frac{H}{w}\right)_\theta\frac{w_\theta}{w}-(n-1)\lambda^2\tan\theta \frac{H_\theta}{H}\\
{}&-\left(\frac{f'}{f}-(n-1)\lambda^2\right)\,.
\end{align*}
Rewriting
\[
\frac{H_\theta}{H}=\frac{w}{H}\left(\frac{H}{w}\right)_\theta-\frac{w_\theta}{w}\quad\text{and}\quad
\frac{w_\theta}{w}=-\tan\theta
\]
we obtain
\begin{align}\label{eq:evolveH/w}
\frac{w}{H}\left(\partial_t-\kappa^2\partial^2 _{\theta}\right)\frac{H}{w}%={}&\frac{w}{H}\left(\frac{H}{w}\right)_\theta \left(2\frac{w_\theta}{w}-(n-1)\lambda^2\tan\theta\right)\\
%{}&-(n-1)\lambda^2\tan\theta\frac{w_\theta}{w}-\left(\frac{f'}{f}-(n-1)\lambda^2\right)\\
+\tan\theta\frac{w}{H}\left(\frac{H}{w}\right)_\theta &\left((n-1)\lambda^2+2\right)\nonumber\\
={}&(n-1)\sec^2\theta\lambda^2-\frac{f'}{f}\nonumber\\
%{}&+(n-1)\frac{\lambda^2}{\cos^2\theta}-\frac{f'}{f}
={}&\frac{(n-1)}{y^2}-\frac{f'}{f}\geq\frac{(n-1)}{\ell^2}-\frac{f'}{f}\,.
\end{align}
%Replacing $\frac{w_\theta}{w}$ by $-\tan\theta$, the reaction term becomes
%\begin{align*}
%-(n-1)\lambda^2\tan\theta\frac{w_\theta}{w}-\left(\frac{f'}{f}-(n-1)\lambda^2\right)={}&
%(n-1)\frac{\lambda^2}{\cos^2\theta}-\frac{f'}{f}={}&
%\frac{(n-1)}{y^2}-\frac{f'}{f}\geq\frac{(n-1)}{\ell^2}-\frac{f'}{f}\,.
%\end{align*}
Now fix any $\varepsilon\in(0,1)$. Then, by Lemma \ref{lem:diameterbounds}, there is some $C<\infty$ such that
\[
\frac{1}{\ell(t)^2}\geq\frac{1}{(-t)^2}-\frac{C}{(-t)^{3-\varepsilon}}
\]
for all $t\in(-\infty,-1]$, say. Thus, if we set%, for any $\varepsilon>0$ and $C<\infty$,
\[
f(t):=\exp\left((n-1)\left[\frac{1}{-t}-\frac{C}{(2-\varepsilon)(-t)^{2-\varepsilon}}\right]\right)
\]
so that
\[
\frac{f'(t)}{f(t)}=(n-1)\left[\frac{1}{(-t)^2}-\frac{C}{(-t)^{3-\varepsilon}}\right]\leq \frac{n-1}{\ell(t)^2}
\]
for all $t<-1$ then, by the maximum principle,
\[
\min_{S^1\times\{t\}}\frac{H}{w}\geq\min_{S^1\times\{t_0\}}\frac{H}{w}\quad\text{for all}\quad -1>t>t_0\,.
\]
But, by Corollary \ref{cor:asymptoticsedge}, the right hand side approaches 1 as $t\to-\infty$. The claim follows by estimating $\exp(\zeta)\geq 1+\zeta$.
\end{proof}

Integrating the lower speed bound yields a displacement bound.
\begin{lemma}\label{lem:ellasymptotics}
The limit
\[
C:=\lim_{t\to-\infty}(\ell(t)+t-(n-1)\log(-t))
\]
exists (in the extended real line $\R\cup\{\infty\}$).
\end{lemma}
\begin{proof}
Given any $\varepsilon\in(0,1)$ set $f(t):=\frac{C}{1-\varepsilon}\frac{1}{(-t)^{1-\varepsilon}}$ for some $C\in\R$. By Corollary \ref{cor:improvedHestimate}, we can choose $C$ so that
\[
\frac{d}{dt}(\ell+t-(n-1)\log(-t)-f)\leq 0\,.
\]
The claim follows because $\lim_{t\to-\infty}f=1$.
\end{proof}

\begin{remark}\label{rem:betterasymptotics}
At this point, it should be possible to obtain higher order terms in the asymptotic expansion for $\ell$ by passing the better estimate $\ell(t)\sim -t+(n-1)\log(-t)$ back into the machine of Lemma \ref{lem:uHhk}. This may be useful for some applications; however, it will not be needed to obtain the uniqueness result of Theorem \ref{thm:uniqueness}.
\end{remark}

\section{Uniqueness}\label{sec:uniqueness}

We are now almost ready to prove uniqueness of the solution constructed in Theorem \ref{thm:limit}. But first we need to show that the limit
\begin{equation}\label{eq:pancakedisplacement}
C:=\lim_{t\to-\infty}(\ell(t)+t-(n-1)\log(-t))
\end{equation}
is finite on the particular solution that we have constructed. To achieve this, we will obtain area and displacement bounds for its approximating solutions by the methods developed in Section~\ref{sec:area estimates}. The precise knowledge of the initial data for the approximating solutions will then allow us to show that $C$ is finite.%To achieve this we need to first show that the area estimates of Corollary \ref{cor:areabounds} hold for the approximating solutions, $\{\Sigma^R_t\}_{t\in[-T_R,0)}$. %The area estimates of Corollary \ref{cor:areabounds} are a consequence of the estimate for the integral of $\l$ provided in Lemma \ref{lem:lupperbound}. Therefore it suffices to prove that Lemma \ref{lem:lupperbound} is true for the sequence of solutions to \eqref{SR-MCF}.
%The only thing in the proof of Lemma \ref{lem:lupperbound} that does not work for these solutions is the fact that the mean curvature is monotone in time (as these solutions are not ancient) and thus \eqref{evolutionofu} is not a consequence of the Harnack inequality anymore. Nevertheless, we will show that \eqref{evolutionofu}  is still true by using the maximum principle instead.

Note first that the Harnack inequality does not yield the necessary lower speed bound $H_R\geq \vert{\cos\theta}\vert$ along the approximating solutions. By construction, the desired estimate does hold at the initial time $t=-T_R$, however, since, by \eqref{k}, $H_R(\theta,-T_R)\geq \k_R(\theta, -T_R)>\vert{\cos\theta}\vert$. It also holds trivially at the points $\{\pm\frac{\pi}{2}\}\times[-T_R,0)$. Setting $f\equiv 0$ in \eqref{eq:evolveH/w}, the maximum principle then yields
\[
H_R(\theta,t)\geq \vert{\cos\theta}\vert\quad\text{for all}\quad (\theta,t)\in S^1\times[-T_R,0)\,.
\]

This estimate, combined with Lemma \ref{CR-lem}, allows us to proceed similarly as in Section \ref{sec:area estimates} to obtain an analogue of Corollary \ref{cor:areabounds}. Namely,
\begin{equation}\label{eq:CRareas}
-t+(n-1)\log(-t)+C\ge \frac{A_R(t)}{2\pi}\ge-t+(n-1)\log(-t)-C
\end{equation}
for all $t\in[-T_R,-1)$, where $C<\infty$ is a constant which does not depend on $R$. In particular, bounding $A_R$ from above by the area of the region of the slab between the lines $y=\pm \ell_R$, we obtain the lower bound
\[
\ell_R(t)\geq-t+(n-1)\log(-t)-C\quad\text{for all}\quad t\in [-T_R,-1)\,.
\]
Since $C$ does not depend on $R$, this estimate passes to the limit solution. So it remains to bound $\ell_R$ from above (uniformly in $R$). Proceeding as in the proof of Lemma \ref{lem:diameterbounds}, we also obtain, for any $\varepsilon\in(0,1)$,
\[
\ell_R(t)\leq-t+C(-t)^\varepsilon\quad\text{for all}\quad t\in[-T_R,-1)
\,,\]
where $C$ depends on $\varepsilon$ but not on $R$. Recalling \eqref{k} and Corollary \ref{time}, we now find that, for $\theta\in(-\frac{\pi}{2},\frac{\pi}{2})$,
\[
\frac{H_R(\theta,-T_R)}{\cos\theta}\geq \frac{1}{\sqrt{1-\mathrm{e}^{-2R}}}+\frac{n-1}{R}\geq 1+\frac{n-1}{T_R}+o\left(\frac{1}{T_R^{2-\varepsilon}}\right)
\]
for any $\varepsilon\in(0,1)$. The maximum principle, applied as in Corollary \ref{cor:improvedHestimate}, then yields, for any $\varepsilon\in(0,1)$,
\[
H_R(\theta,t)\geq \vert{\cos\theta}\vert\left(1+\frac{n-1}{-t}-\frac{C}{(-t)^{2-\varepsilon}}\right)
\]
for all $(\theta,t)\in S^1\times [-T_R,-1)$, where $C<\infty$ depends on $\varepsilon$ but not on $R$. Thus, as in Lemma \ref{lem:ellasymptotics}, given any $\varepsilon\in(0,1)$ we may choose $C<\infty$ so that
\[
\frac{d}{dt}(\ell_R+t-(n-1)\log(-t)-f)\leq 0\quad \text{for all}\quad t\in[-T_R,-1)\,,
\]
where $f(t):=\frac{C}{1-\varepsilon}\frac{1}{(-t)^{1-\varepsilon}}$. Integrating yields
\begin{align*}
\ell_R(t)+t-(n-1)\log(-t)\leq{}& \ell_R(-T_R)-T_R-(n-1)\log(-T_R)\\
{}&+f(t)-f(-T_R)
\end{align*}
for all $ t\in[-T_R,0)$. But the initial area is given exactly by $A_R(-T_R)=2\pi R$ and the initial displacement $\ell_R(-T_R)$ is comparable to $R$ so that, recalling \eqref{eq:CRareas},
\[
%T_R+(n-1)\log T_R-C\le 
\ell_R(-T_R)\le T_R+(n-1)\log T_R+C+\log 2\,.% \quad\text{for all}\quad R>0\,.
\]
Since $\lim_{t\to-\infty}f(t)=1$ and $C$ is independent of $R$, taking $R\to\infty$ we conclude that the limit in \eqref{eq:pancakedisplacement} is finite.

We now have all the required estimates in place needed to prove uniqueness of the solution constructed in Theorem \ref{thm:existence}.

\begin{proof}[Proof of Theorem \ref{thm:uniqueness}]
Let $\{\Sigma_t\}_{t\in(-\infty,0)}$ be the solution constructed in Theorem \ref{thm:limit} and let $\{\Sigma'_t\}_{t\in(-\infty,0)}$ be another compact, convex $O(n)$-invariant ancient solution that lies in the slab $\Omega:=\{(x,y,z)\in \R^{n+1}:-\tfrac{\pi}{2}<x<\tfrac{\pi}{2}\}$ (and in no smaller slab). Note that, by Theorem \ref{thm:reflection symmetry}, the latter solution is also reflection symmetric about the hyperplane $\{ x_1=0\}$. Let $\gamma:S^1\times(-\infty,0)\to\E^2$ and $\gamma':S^1\times(-\infty,0)\to \E^2$ be turning angle parametrizations of the profile curves $\{\Gamma_t\}_{t\in(-\infty,0)}$ and $\{\Gamma'_t\}_{t\in(-\infty,0)}$ of the two solutions, respectively. Recall that the limits {\small
\begin{equation*}
C:=\lim_{t\to-\infty}(\ell(t), e_2\ra+t-\log(-t))\;\;\text{and}\;\; C':=\lim_{t\to-\infty}(\ell'(t), e_2\ra+t-\log(-t))
\end{equation*}
}exist and that $C$ is finite, where $\ell(t)$ and $\ell'(t)$ are the vertical displacements of $\Gamma_t$ and $\Gamma'_t$ respectively. We claim that $C=C'$. Indeed, if $C\ne C'$ then there must be some $t_0<0$ such that, for all $t<t_0$, either $\ell(t)<\ell'(t) \quad\text{or}\quad\ell'(t)<\ell(t)$. %Given $\alpha\in[0,\tfrac\pi2)$, define the halfspace $\Pi_\alpha:\{(x,y,z)\in\R^{n+1}:x<\alpha\}$. Denote the reflection of $\Sigma_t$ in the hyperplane $\partial\Pi_\alpha$ by $R_\alpha\cdot\Sigma_t:=\{(2\alpha-x,y,z):(x,y,z)\in \Sigma_t\}$ and 
Recalling the notation of Section \ref{sec:reflection}, set $\wt \Sigma^{\a}_{t}:=(R_\a\cdot\Sigma_{t})\cap \Pi_0$ for any $\alpha\in(0,\frac{\pi}{2})$ and $\wt \Sigma'_{t}:=\Sigma'_{t}\cap \Pi_0$. Then $\partial\wt \Sigma^{\a}_{t}\cap\partial\wt \Sigma'_{t}=\emptyset$ for all $t<t_0$. Moreover, applying Corollary \ref{cor:asymptoticsedge} as in Claim \ref{claim:reflectionII}, we can find, for each $\alpha\in(0,\frac{\pi}{2})$, some $t_\alpha>\infty$ such that $\wt \Sigma^{\a}_{t}\cap \wt \Sigma'_{t}=\emptyset$ for all $t<t_\a$. It now follows from the strong maximum principle (cf. \cite[Theorem 2.2]{Chow97}) that $\wt \Sigma^{\a}_{t}\cap \wt \Sigma_{t}=\emptyset$ for all $t<t_0$. Taking $\alpha\to 0$, we see that either $\Omega_t\subset\Omega_t'$ or $\Omega'_t\subset\Omega_t$ for all $t<t_0$, where $\Omega_t$ and $\Omega'_t$ are the domains bounded by $\Sigma_t$ and $\Sigma'_t$, respectively. But, by the avoidance principle, $\Sigma_t$ and $\Sigma'_t$ must intersect for all $t<0$, so we conclude from the strong maximum principle that $\Sigma_t=\Sigma_t'$ for all $t<t_0$ and hence for all $t<0$.

So we can assume that $C=C'$. Consider, for any $\tau>0$, the solution $\Sigma^\tau_t:=\Sigma_{t+\tau}$. The vertical displacement $\ell^\tau$ of this solution's profile curve satisfies
\[
C^\tau:=\lim_{t\to-\infty}(\ell^\tau(t)+t-\log(-t))>C
\]
and hence $C^\tau>C'$. Arguing as above, we conclude that the regions $\Omega^\tau_t$ bounded by $\Sigma^\tau_t$ satisfy $\Omega'_t\subset\Omega^\tau_t$. Taking $\tau\to 0$, we conclude that the two solutions must be identical.
\end{proof}
 
%\center{Giuseppe Tinaglia at giuseppe.tinaglia@kcl.ac.uk\\ Department of Mathematics, King's College London, London, WC2R 2LS, U.K.}
\bibliographystyle{acm}
\bibliography{pancakes}

\end{document}